\documentclass[12pt,a4paper]{article}
\pdfoutput=1
\usepackage{graphicx}
\usepackage[]{amsmath,amssymb,amsthm}
\usepackage{paralist}
\usepackage{epstopdf}
\usepackage[section]{placeins}
\usepackage[utf8]{inputenc}
\usepackage[T1]{fontenc}
\usepackage{tikz}
\usetikzlibrary{decorations.pathreplacing,decorations.markings,positioning}
\usepackage{subfigure}
\usepackage{hyperref}

\tikzset{
    middlearrow/.style n args={3}{
        draw,
        decoration={
            markings,
            mark=at position 0.5 with {
                \arrow[scale=2]{#1};
                \path[#2] node {$#3$};
            },
        },
        postaction=decorate
    }
}

\newtheorem{theorem}{Theorem}[section]
\newtheorem{proposition}[theorem]{Proposition}

\theoremstyle{definition}
\newtheorem{definition}[theorem]{Definition}
\theoremstyle{remark}
\newtheorem{remark}[theorem]{Remark}
\newcommand{\R}{\mathbb{R}}
\newcommand{\N}{\mathbb{N}}

\newcommand{\ind}{\textnormal{index}}
\newcommand{\diag}{\textnormal{diag}}
\newcommand{\B}{{\cal B}}

\begin{document}

\begin{center}
{\large{\bf Arbitrarily large heteroclinic networks in fixed low-dimensional state space}}\\
\mbox{} \\
\begin{tabular}{cc}
{\bf Sofia B.\ S.\ D.\ Castro$^{\dagger}$} & {\bf Alexander Lohse$^{\ddagger,*}$} \\
{\small sdcastro@fep.up.pt} & {\small alexander.lohse@uni-hamburg.de}
\end{tabular}
\end{center}

\noindent $^{*}$ Corresponding author.

\noindent $^{\dagger}$ Faculdade de Economia and Centro de Matem\'atica, Universidade do Porto, Rua Dr.\ Roberto Frias, 4200-464 Porto, Portugal.

\noindent $^{\ddagger}$ Fachbereich Mathematik, Universit\"at Hamburg, Bundesstra{\ss}e 55, 20146 Hamburg, Germany.

\begin{abstract}
We consider heteroclinic networks between $n \in \N$ nodes where the only connections are those linking each node to its two subsequent neighbouring ones. Using a construction method where all nodes are placed in a single one-dimensional space and the connections lie in coordinate planes, we show that it is possible to robustly realise these networks in $\R^6$ for any number of nodes $n$ using a polynomial vector field. This bound on the space dimension (while the number of nodes in the network goes to $\infty$) is a novel phenomenon and a step towards more efficient realisation methods for given connection structures in terms of the required number of space dimensions. We briefly discuss some stability properties of the generated heteroclinic objects. 
\end{abstract}


\vspace{.3cm}


{\bf In a dynamical system there may be sequences of equilibria connected by solutions, called heteroclinic trajectories. These can form a heteroclinic cycle, when finitely many equilibria and connecting heteroclinic trajectories form a closed loop, or a heteroclinic network, when a finite number of heteroclinic cycles have a non-empty intersection. Although the graph of a heteroclinic network is unique, the vector fields that can support the existence of a heteroclinic network (corresponding to a given graph) are not. Existing work on finding these vector fields is such that an increase in the number of connecting trajectories and/or equilibria demands an increase in state space dimension. We prove that this need not be the case and provide a construction process that works in state space dimension at most 6.}

\section{Introduction}

We are concerned with heteroclinic networks between nodes $\xi_1, \ldots, \xi_n$ with connections $[\xi_k \to \xi_{k+1}]$ and $[\xi_k \to \xi_{k+2}]$ for all $k$ (modulo $n$). We call these networks \emph{double-next-neighbour (DNN) networks}. Well-known examples of DNN networks are the Rock-Paper-Scissors with two players \cite{GdS-C2020} and the Rock-Paper-Scissors-Lizard-Spock \cite{Castro_etal2022, PostlethwaiteRucklidge2022}. These exhibit cyclic dominance which is a feature of DNN networks where connections have an inhibitory effect.
 Such a connection structure can be realised as a heteroclinic network in $\R^n$ through one of several construction methods in the literature \cite{AP2013, AP2016a, Field2015, Field2017}. In a discrete-time setting, DNN networks appear in \cite{PosStu2022} as ring graphs of type $(n,2)$, that is, $n$ nodes with $2$-nearest-neighbour coupling. The smallest such graph is the aforementioned Rock-Paper-Scissors-Lizard-Spock. 
Afraimovich et al.\ \cite{AfraimovichMosesYoung} use a realisation for DNN networks that is in spirit similar to the simplex method \cite{AP2013}: they place each node on its own coordinate axis, thus requiring at least as many space dimensions as there are nodes. In this paper we give an alternative construction for heteroclinic networks that realise the same connection structure in $\R^N$, such that the space dimension $N$ is independent of the number of nodes in the network $n$. This means we can fit any such network into a space of fixed dimension, regardless of how many nodes the network involves. We are not aware of similar results in the literature.

Our idea is to use a construction that is geometrically similar to the cylinder method \cite{AP2013}, even though the equations turn out to be substantially different. We construct a polynomial vector field in $\R^N$, $N\leq 6$, with a degree that is increasing in the number of nodes $n$. Geometrically, all nodes are placed on the same axis and connections are contained in two-dimensional coordinate planes. However, we crucially make efficient use of these planes by placing more and more connections in the same plane as the number of nodes increases: the cylinder method uses a new plane for each connection, thus requiring $2n$ planes in total for our type of network, realising it in $\R^{2n+1}$ because of the additional dimension containing the nodes.

We provide a more efficient construction than that in two ways:
\begin{itemize}
\item by using each plane twice, i.e.\ by placing $[\xi_{k-2} \to \xi_k]$ and $[\xi_{k-1} \to \xi_k]$, the incoming connections to $\xi_{k}$, in the same plane $P_{0k}$, see the coloring\footnote{We use different color and line type for the convenience of those reading in black and white.} of connections in Figure~\ref{fig1}, where the case $n=4$ is illustrated.
\item by placing the connections $[\xi_{k-2} \to \xi_k]$ and $[\xi_{k-1} \to \xi_k]$ in the same plane for different values of $k$.
\end{itemize}

We define coordinates $(x,y_1,\ldots,y_{N-1})$ in $\R^N$ and write $P_{0k}$ for the two-dimensional subspace where only $x$ and $y_k$ are allowed to be non-zero. For our results $N\leq 6$ suffices for the construction of a DNN network with any number of nodes. 

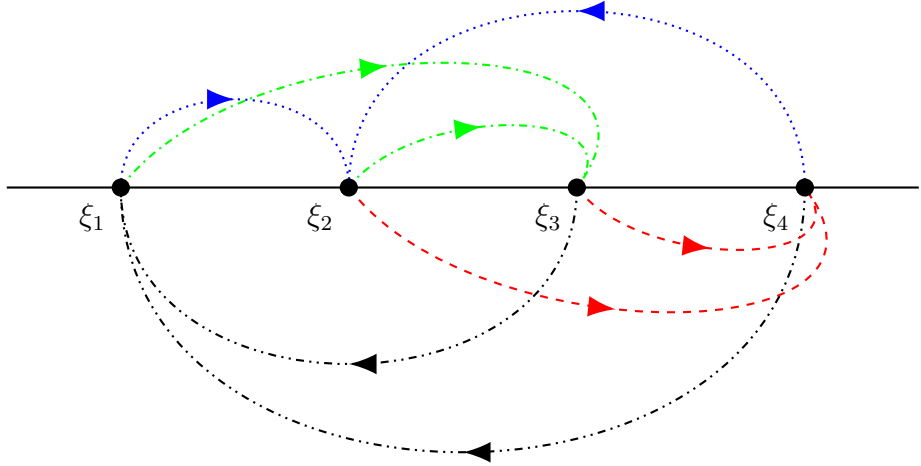
\begin{figure}[!htb]
 \centering
 \begin{tikzpicture}[>=latex,thick,xscale=1.5,yscale=2]
\coordinate (1) at (0,0);
\coordinate (2) at (2,0);
\coordinate (3) at (4,0);
\coordinate (4) at (6,0);
\draw (-1,0) -- (7,0);
\path (3) edge[dash dot dot, out=-90,in=-90,middlearrow={>}{}{}] (1);
\path (4) edge[dash dot dot, out=-90,in=-90,middlearrow={>}{}{}] (1);
\path (1) edge[blue, dotted, out=90,in=90,middlearrow={>}{}{}] (2);
\path (4) edge[blue, dotted, out=90,in=90,middlearrow={>}{}{}] (2);
\path (1) edge[green, dash dot, out=45,in=45,middlearrow={>}{}{}] (3);
\path (2) edge[green, dash dot, out=45,in=45,middlearrow={>}{}{}] (3);
\path (2) edge[red, dashed, out=-45,in=-45,middlearrow={>}{}{}] (4);
\path (3) edge[red, dashed, out=-45,in=-45,middlearrow={>}{}{}] (4);
\filldraw[black,yscale=0.75] (1) circle (2pt) node[below left=.5mm and .5mm of 1] {$\xi_1$};
\filldraw[black,yscale=0.75] (2) circle (2pt) node[below left=.5mm and .5mm of 2] {$\xi_2$};
\filldraw[black,yscale=0.75] (3) circle (2pt) node[below left=.5mm and .5mm of 3] {$\xi_3$};
\filldraw[black,yscale=0.75] (4) circle (2pt) node[below left=.5mm and .5mm of 4] {$\xi_4$};
\end{tikzpicture}
 \caption{Sketch of a DNN network for $n=4$ in $\R^5$. Four different planes are used, each containing the heteroclinic connections incoming at one of the four nodes. An additional line contains all the equilibria. Connections with the same colour or the same line type occur in the same plane.}
 \label{fig1}
\end{figure}

For a small number of nodes ($n \leq 5$) it is not possible to have in the same plane the connections $[\xi_{k-2} \to \xi_k]$ and $[\xi_{k-1} \to \xi_k]$ as well as the connections $[\xi_{j-2} \to \xi_j]$ and $[\xi_{j-1} \to \xi_j]$ for $j \neq k$, but we show how it can be done for $n \geq 6$, see Figure~\ref{fig2}, where $n=8$ is illustrated. Our main result Theorem~\ref{thm1} states that with this technique any directed graph of the type described above can be realised as a heteroclinic network in $\R^6$ by a polynomial vector field.

We note that our realisation is also more efficient than the simplex method when the number of nodes is greater than six.

Our construction is well suited to the context of coupled cell networks. In such context, the line containing all the equilibria corresponds to a synchrony subspace of all the cells and the connections correspond to a smaller number of cells desynchronising from the rest to resynchronise at the end of each connection. See, for instance, Field \cite{Field2015}.

Furthermore, our construction is connected to low-dimensional realisations for particular forms of {\em winnerless competition} networks. The concept of winnerless competition was introduced by \cite{Rab_etal2001} and appears, among other sources, in \cite{AfrRabVar2004}, \cite{AfrZakRab2004} and \cite{AfrZhiRab2004}.

We note that our DNN networks are not nearest and next-to-nearest neighbour networks as normally used in the context of coupled cells. The connections from a node are only to nodes with higher indices (modulo $n$). Therefore, e.g., in a DNN network $\xi_3$ does not connect to $\xi_2$ unless there are only three nodes in the network.

The rest of this paper is structured as follows: in section~\ref{sec_prelim} we introduce relevant terminology and results from the literature. Then in section~\ref{sec_3-6} we discuss the cases $n=3,4,5,6$ in detail, which enables us to state and prove the general case in section~\ref{sec_n}. Finally, in section~\ref{sec_stability} we comment on the stability of some of the cycles generated in the networks we design.

\section{Preliminaries}
\label{sec_prelim}
For a (smooth) continuous dynamical system on $\R^N$ given by
\begin{align}
\label{sys}
\dot x=f(x)
\end{align}
a \emph{heteroclinic cycle} consists of finitely many equilibria $\xi_1, \ldots,\xi_n$ together with trajectories that connect them cyclically in the sense that
$$W^u(\xi_k) \cap W^s(\xi_{k+1}) \neq \emptyset$$
for all $k$ modulo $n$. A \emph{heteroclinic network} is a connected union of more than one, and finitely many, heteroclinic cycles. Slightly deviating definitions can be found in the literature, but the differences are of no concern for the purpose of this paper. 

A heteroclinic cycle or network can intuitively be associated with a directed graph (digraph), where the vertices and edges of the digraph correspond to equilibria and connections in the cycle or network, respectively. We say that a given digraph is realised by system~(\ref{sys}) if there is a one-to-one correspondence between vertices/edges in the graph and equilibria / heteroclinic connections for the dynamics of system~(\ref{sys}), see \cite{AshwinCastroLohse, PodviginaLohse2019} for a precise discussion of this topic.

In this work we are concerned with a particular type of digraph and its realisations as heteroclinic networks:
\begin{definition}
We call a digraph with vertices $v_1, \ldots,v_n$ and edges $[v_k \to v_{k+1}]$ and $[v_k \to v_{k+2}]$ for all $k$ (modulo $n$) a \emph{double-next-neighbour graph} or \emph{DNN graph}.

Furthermore, we call a heteroclinic network realising a given DNN graph a \emph{DNN network}.
\end{definition}

Note that DNN graphs are unique for a given $n$, but the same is not true for DNN networks, since there are many different ways of realising a given digraph as a heteroclinic network: besides the realization method in this article, see alternative realizations in the literature such as \cite{AP2013, AP2016a, Field2015, Field2017}. The DNN graph for $n\in \{3,4,5,6\}$ is shown in Figure~\ref{fig3}, while Figures~\ref{fig1} and \ref{fig2} display sketches of our realisations of the DNN graphs for $n=4$ and $n=8$, respectively.

\begin{figure}
\centering
\begin{tikzpicture}[>=latex,thick,xscale=.9,yscale=1.2]
\coordinate (1) at (0,0);
\coordinate (2) at (2,0);
\coordinate (3) at (4,0);
\coordinate (4) at (6,0);
\coordinate (5) at (8,0);
\coordinate (6) at (10,0);
\coordinate (7) at (12,0);
\coordinate (8) at (14,0);
\draw (-0.5,0) -- (14.5,0);
\path (7) edge[dash dot dot, out=-45,in=-45,middlearrow={>}{}{}] (1);
\path (8) edge[dash dot dot, out=-45,in=-45,middlearrow={>}{}{}] (1);
\path (1) edge[blue, dotted, out=45,in=45,middlearrow={>}{}{}] (2);
\path (8) edge[blue, dotted, out=45,in=45,middlearrow={>}{}{}] (2);
\path (1) edge[green, dash dot, out=90,in=90,middlearrow={>}{}{}] (3);
\path (2) edge[green, dash dot, out=90,in=90,middlearrow={>}{}{}] (3);
\path (4) edge[green, dash dot, out=90,in=90,middlearrow={>}{}{}] (6);
\path (5) edge[green, dash dot, out=90,in=90,middlearrow={>}{}{}] (6);
\path (2) edge[red, dashed, out=-90,in=-90,middlearrow={>}{}{}] (4);
\path (3) edge[red, dashed, out=-90,in=-90,middlearrow={>}{}{}] (4);
\path (5) edge[red, dashed, out=-90,in=-90,middlearrow={>}{}{}] (7);
\path (6) edge[red, dashed, out=-90,in=-90,middlearrow={>}{}{}] (7);
\path (3) edge[orange, loosely dashed, out=70,in=70,middlearrow={>}{}{}] (5);
\path (4) edge[orange, loosely dashed, out=70,in=70,middlearrow={>}{}{}] (5);
\path (6) edge[orange, loosely dashed, out=70,in=70,middlearrow={>}{}{}] (8);
\path (7) edge[orange, loosely dashed, out=70,in=70,middlearrow={>}{}{}] (8);
\filldraw[black,yscale=0.75] (1) circle (2pt) node[below left=.5mm and .5mm of 1] {$\xi_1$};
\filldraw[black,yscale=0.75] (2) circle (2pt) node[below left=.5mm and .5mm of 2] {$\xi_2$};
\filldraw[black,yscale=0.75] (3) circle (2pt) node[below left=.5mm and .5mm of 3] {$\xi_3$};
\filldraw[black,yscale=0.75] (4) circle (2pt) node[below left=.5mm and .5mm of 4] {$\xi_4$};
\filldraw[black,yscale=0.75] (5) circle (2pt) node[below left=.5mm and .5mm of 5] {$\xi_5$};
\filldraw[black,yscale=0.75] (6) circle (2pt) node[below left=.5mm and .5mm of 6] {$\xi_6$};
\filldraw[black,yscale=0.75] (7) circle (2pt) node[below left=.5mm and .5mm of 7] {$\xi_7$};
\filldraw[black,yscale=0.75] (8) circle (2pt) node[below left=.5mm and .5mm of 8] {$\xi_8$};
\end{tikzpicture}
 \caption{Sketch of a DNN network for $n=8$ in $\R^6$. The incoming connections at $\xi_j$ and $\xi_{j+3}$ are both contained in the plane $P_{0j}$ for $j=3,4,5$. Connections with the same colour or the same line type occur in the same plane.}
\label{fig2}
\end{figure}
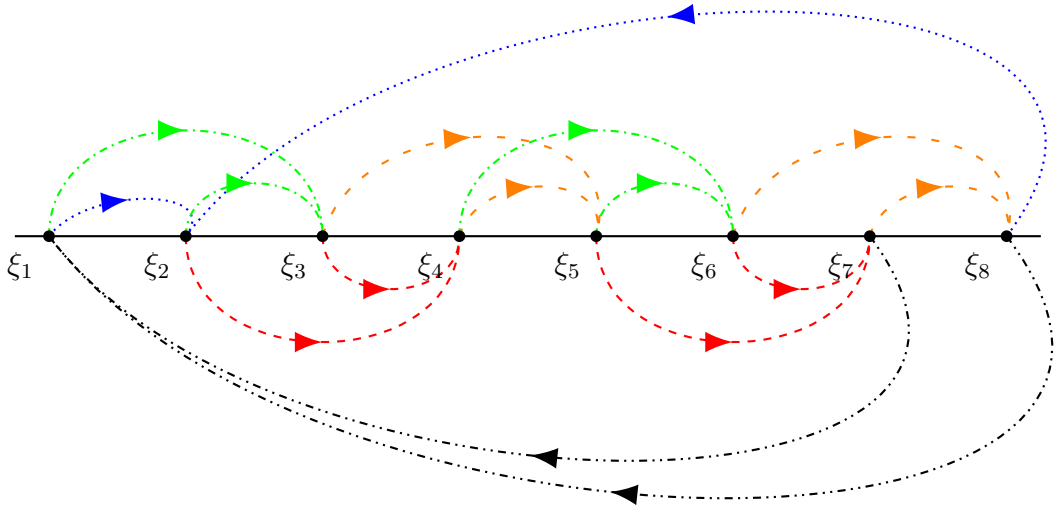

\begin{figure}
\centering
\subfigure{
	\begin{tikzpicture}[>=latex,thick,scale=1.2]
	\coordinate (1) at (0,0);
	\coordinate (2) at (4,0);
	\coordinate (3) at (2,-3.464);
	\filldraw[black] (1) circle (2pt) node[left] {$\xi_1$};
	\filldraw[black] (2) circle (2pt) node[right] {$\xi_2$};
	\filldraw[black] (3) circle (2pt) node[below] {$\xi_3$};
	\path (1) edge[middlearrow={>}{left}{}, bend right] (2);
	\path (2) edge[middlearrow={>}{left}{}, bend right] (3);
	\path (3) edge[middlearrow={>}{left}{}, bend right] (1);
	\path (2) edge[middlearrow={>}{left}{}, bend right] (1);
	\path (3) edge[middlearrow={>}{left}{}, bend right] (2);
	\path (1) edge[middlearrow={>}{left}{}, bend right] (3);
	\end{tikzpicture}
}
\qquad
\subfigure{
	\begin{tikzpicture}[>=latex,thick,scale=1.2]
	\coordinate (1) at (0,0);
	\coordinate (2) at (4,0);
	\coordinate (3) at (4,-4);
	\coordinate (4) at (0,-4);
	\filldraw[black] (1) circle (2pt) node[left] {$\xi_1$};
	\filldraw[black] (2) circle (2pt) node[right] {$\xi_2$};
	\filldraw[black] (3) circle (2pt) node[right] {$\xi_3$};
	\filldraw[black] (4) circle (2pt) node[left] {$\xi_4$};
	\path (1) edge[middlearrow={>}{left}{}] (2);
	\path (2) edge[middlearrow={>}{left}{}] (3);
	\path (3) edge[middlearrow={>}{left}{}] (4);
	\path (4) edge[middlearrow={>}{left}{}] (1);
	\path (1) edge[out=340,in=110,middlearrow={>}{}{}] (3);
	\path (3) edge[out=160,in=290,middlearrow={>}{}{}] (1);
	\path (2) edge[out=250,in=20,middlearrow={>}{}{}] (4);
	\path (4) edge[out=70,in=200,middlearrow={>}{}{}] (2);
	\end{tikzpicture}
}
\bigbreak
\bigbreak
\subfigure{
	\begin{tikzpicture}[>=latex,thick,scale=3, baseline={(0,-4.5)}]
	\coordinate (1) at (0,0.588);
	\coordinate (2) at (0.809,0);
	\coordinate (3) at (0.5,-0.951);
	\coordinate (4) at (-0.5,-0.951);
	\coordinate (5) at (-0.809,0);
	\filldraw[black] (1) circle (0.8pt) node[above] {$\xi_1$};
	\filldraw[black] (2) circle (0.8pt) node[right] {$\xi_2$};
	\filldraw[black] (3) circle (0.8pt) node[right] {$\xi_3$};
	\filldraw[black] (4) circle (0.8pt) node[left] {$\xi_4$};
	\filldraw[black] (5) circle (0.8pt) node[left] {$\xi_5$};
	\path (1) edge[middlearrow={>}{left}{}] (2);
	\path (2) edge[middlearrow={>}{left}{}] (3);
	\path (3) edge[middlearrow={>}{left}{}] (4);
	\path (4) edge[middlearrow={>}{left}{}] (5);
	\path (5) edge[middlearrow={>}{left}{}] (1);
	\path (1) edge[middlearrow={>}{left}{}] (3);
	\path (2) edge[middlearrow={>}{left}{}] (4);
	\path (3) edge[middlearrow={>}{left}{}] (5);
	\path (4) edge[middlearrow={>}{left}{}] (1);
	\path (5) edge[middlearrow={>}{left}{}] (2);
	\end{tikzpicture}
}
\quad
\subfigure{
	\begin{tikzpicture}[>=latex,thick,scale=1]
	\coordinate (1) at (0,3);
	\coordinate (2) at (2.598,1.5);
	\coordinate (3) at (2.598,-1.5);
	\coordinate (4) at (0,-3);
	\coordinate (5) at (-2.598,-1.5);
	\coordinate (6) at (-2.598,1.5);
	\filldraw[black] (1) circle (2.4pt) node[above] {$\xi_1$};
	\filldraw[black] (2) circle (2.4pt) node[right] {$\xi_2$};
	\filldraw[black] (3) circle (2.4pt) node[right] {$\xi_3$};
	\filldraw[black] (4) circle (2.4pt) node[below] {$\xi_4$};
	\filldraw[black] (5) circle (2.4pt) node[left] {$\xi_5$};
	\filldraw[black] (6) circle (2.4pt) node[left] {$\xi_6$};
	\path (1) edge[middlearrow={>}{left}{}] (2);
	\path (2) edge[middlearrow={>}{left}{}] (3);
	\path (3) edge[middlearrow={>}{left}{}] (4);
	\path (4) edge[middlearrow={>}{left}{}] (5);
	\path (5) edge[middlearrow={>}{left}{}] (6);
	\path (6) edge[middlearrow={>}{left}{}] (1);
	\path (1) edge[middlearrow={>}{left}{}] (3);
	\path (2) edge[middlearrow={>}{left}{}] (4);
	\path (3) edge[middlearrow={>}{left}{}] (5);
	\path (4) edge[middlearrow={>}{left}{}] (6);
	\path (5) edge[middlearrow={>}{left}{}] (1);
	\path (6) edge[middlearrow={>}{left}{}] (2);
	\end{tikzpicture}
}
 \caption{The DNN graphs for $n \in \{3,4,5,6\}$.}
\label{fig3}
\end{figure}
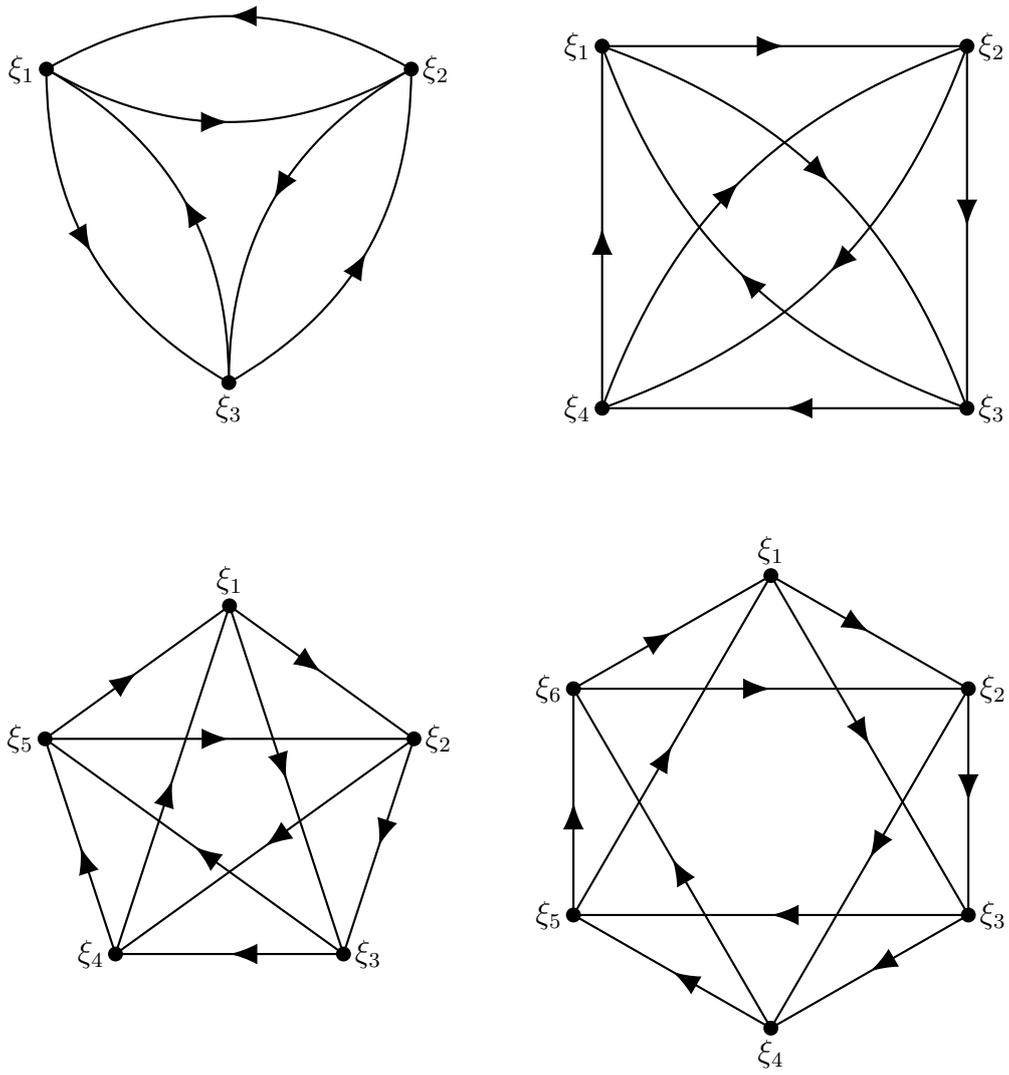
    
In a dynamical system with a DNN network there will often be more steady states than those involved in the heteroclinic structure corresponding to the digraph. To make a clear distinction between these two types of steady states, we refer to those corresponding to vertices in the DNN graph as \emph{nodes} and to the others simply as \emph{equilibria}.

\section{DNN networks with no more than six nodes}
\label{sec_3-6}
In this section we explicitly construct and discuss DNN networks for $n \in \{3,4,5,6\}$ to highlight the strategy used to overcome the main obstacles that occur when many heteroclinic connections are to be realised in the same plane. Looking at these low-dimensional cases in detail provides us with all the tools we need to prove our result for the general case of arbitrary $n \in \N$ in section~\ref{sec_n}.

Using coordinates $(x,y_1,\ldots,y_{N-1})$ in $\R^N$, $N\leq 6$, we write $\R^N_{y\geq 0}$ for the part of $\R^N$, where all $y_j$-coordinates are non-negative. Our vector fields are defined on $\R^N_{y \geq 0}$, but can be symmetrically extended to all of $\R^N$. However, for many cases of interest (e.g.\ population dynamics or game theory) this is not necessary.

Recall that our construction places all nodes on the $x$-axis, so each node $\xi_k$ has a unique non-zero coordinate. For convenience of notation we sometimes identify $\xi_k$ with its non-zero coordinate. So if the $x$-coordinate of $\xi_k$ is, say, $3$ we write $\xi_k + \frac{1}{2}$ instead of $\frac{7}{2}$. This clarifies the placement of the nullclines in what follows.

\subsection{A DNN network with three nodes}\label{subsec:n=3}
We start with the case $n=3$, for which we need four dimensions.
\begin{proposition}
\label{prop_3}
For $\varepsilon>0$ sufficiently small, system~(\ref{sys3}) realises the DNN graph with $n=3$ as a heteroclinic network in $\R^4_{y\geq 0}$ between the nodes $\xi_1:=(1,0,0,0)$, $\xi_2:=(3,0,0,0)$ and $\xi_3:=(5,0,0,0)$.
\begin{align}
\label{sys3}
\begin{cases}
\dot{x} &=-\varepsilon \prod\limits_{k=1}^5 (x-k)\\
&+y_1\Big(-y_1^2-x+\xi_1\Big)\Big(y_1^2+(x-(\xi_2-\frac{1}{2}))^2-\frac{1}{4}\Big)\Big(y_1^2+(x-(\xi_3-\frac{1}{2}))^2-\frac{1}{4}\Big)\\
&+ y_2\Big(-y_2^2-x+\xi_2\Big)\Big(y_2^2+(x-(\xi_1+\frac{1}{2}))^2-\frac{1}{4}\Big)\Big(y_2^2+(x-(\xi_3-\frac{1}{2}))^2-\frac{1}{4}\Big)\\
&+ y_3\Big(-y_3^2-x+\xi_3\Big)\Big(y_3^2+(x-(\xi_1+\frac{1}{2}))^2-\frac{1}{4}\Big)\Big(y_3^2+(x-(\xi_2+\frac{1}{2}))^2-\frac{1}{4}\Big)\\
\dot{y}_1 &= -y_1\Big(y_1^2-x+(\xi_1+\frac{1}{2})\Big)\\
\dot{y}_2 &= y_2\Big(-y_2^2-x+(\xi_2-\frac{1}{2})\Big)\Big(y_2^2-x+(\xi_2+\frac{1}{2})\Big)\\ 
\dot{y}_3 &= y_3\Big(-y_3^2-x+(\xi_3-\frac{1}{2})\Big)
\end{cases}
\end{align}
\end{proposition}

\begin{proof}
The $x-$axis is invariant and the system has equilibria at $(x,0,0,0)$ with $x \in \{1,2,3,4,5\}$. Of these, the nodes are those with $x=1,3,5$, alternating with additional equilibria at $x = 2,4$. There are no equilibria outside the $x$-axis. Recall that, for $j=1,2,3$, we denote the plane spanned by $x$ and $y_j$ as $P_{0j}$. By construction, $P_{0j}$ is flow-invariant for all $j$. We want to show the existence of connections:
\begin{itemize}
	\item $[\xi_2 \to \xi_1]$ and $[\xi_3 \to \xi_1]$ in $P_{01}$
	\item $[\xi_3 \to \xi_2]$ and $[\xi_1 \to \xi_2]$ in $P_{02}$
	\item $[\xi_1 \to \xi_3]$ and $[\xi_2 \to \xi_3]$ in $P_{03}$
\end{itemize}

It is necessary to show that for each $j$, the node $\xi_j$ is a sink in $P_{0j}$ while the other two nodes are saddles; the remaining equilibria are unstable and ensure that all the nodes are stable along the $x$-axis. This is done by studying the $x$- and $y_j$-nullclines of system~(\ref{sys3}) in the respective planes. They are shown in Figure~\ref{fig:n=3}.

For the analysis of the nullclines we ignore the $\varepsilon$-term in the $x$-equation: for $\varepsilon=0$ the $x$-axis is a nullcline that consists only of equilibria and the remaining nullclines are circles/ellipses and parabolas as will be discussed below. For $\varepsilon>0$ but small, the only equilibria on the $x$-axis are those where $x \in \{1,2,3,4,5\}$ and all nullclines from the case with $\varepsilon=0$ are slightly perturbed. Note that this causes the nullcline lying on the $x$-axis for $\varepsilon=0$ to be below the $x$-axis except at the equilibria when $\varepsilon>0$. This is because for $y_j>0$ and small, the remaining terms in the $x$-equation are positive for $x<1$, negative for $x \in (1,2)$, positive for $x \in (2,3)$ and so on. The $\varepsilon$-term has the same sign in those regions and therefore the nullcline is moved below the axis, and thus out of our domain of definition.

By analyzing the Jacobian of the right hand side of system~\eqref{sys3}, the reader can check that all nodes $\xi_k$ have expanding and contracting directions such that they are stable along the $x$-axis while the remaining equilibria are unstable. In each $P_{0j}$, the node $\xi_j$ is a sink while the other two nodes are saddles, i.e.\ they have expanding $y_j$-direction.

We now look at $P_{01}$ and recommend the reader to look at Figure~\ref{fig:n=3} while following our arguments: the $x$-nullclines other than the $x$-axis are\footnote{In writing the vector fields in subsections~\ref{subsec:n=3}--\ref{subsec:n=6} we use bigger brackets to highlight the expressions that correspond to geometric figures in the drawing of the nullclines.}
\begin{itemize}
	\item a parabola at $\xi_1$ that opens to the left, given by the zeros of $-y_1^2-x+\xi_1$;
	\item a (semi)circle intersecting the $x$-axis at $\xi_2$ and $\xi_2-1$, given by the zeros of $y_1^2+(x-(\xi_2-\frac{1}{2}))^2-\frac{1}{4}$;
	\item a (semi)circle intersecting the $x$-axis at $\xi_3$ and $\xi_3-1$, given by the zeros of $y_1^2+(x-(\xi_3-\frac{1}{2}))^2-\frac{1}{4}$.
\end{itemize}
On these nullclines the vector field is vertical and flows as indicated by the arrows in Figure~\ref{fig:n=3}.

The $y_1$-nullclines are the $x$-axis and a parabola opening to the right at $\xi_1+\frac{1}{2}$, corresponding to the zeros of $(y_1^2-x+(\xi_1+\frac{1}{2}))$. The fact that for $\varepsilon > 0$ the $x$-axis is not a nullcline means that there is only the desired set of equilibria, namely $x \in \{3,4,5,6\}$. In particular, the intersection of the red parabola with the $x$-axis does not produce any additional equilibria.

Now let us follow the unstable manifold of $\xi_2$: it leaves $\xi_2$ above the semicircle and thus moves up and left until it crosses the $y_1$-nullcline horizontally. Subsequently, it moves down and left and has no choice but to connect to $\xi_1$ by the Poincar\'{e}-Bendixson theorem. The reasoning for the connection $[\xi_3 \to \xi_1]$ is exactly the same.

In $P_{02}$ the $x$-nullclines are very much as in $P_{01}$, except that the parabola is now at $\xi_2$. The $y_2$-nullclines are the horizontal axis and two parabolas opening in opposite directions, given by the zeros of $(-y_2^2-x+(\xi_2-\frac{1}{2}))(y_2^2-x+(\xi_2+\frac{1}{2}))$. The reasoning is similar: the unstable manifold of $\xi_1$ first moves up and right until it crosses the red $y_2$-nullcline horizontally and is forced to connect to $\xi_2$ by the $x$-nullcline (blue parabola) attached there. Similarly for $[\xi_3 \to \xi_2]$.

Finally, in $P_{03}$ the arguments work in the same way, as becomes clear from the respective picture in Figure~\ref{fig:n=3}.

Once again, we point out that the nullclines are not exactly semicircles and parabolas, but slightly perturbed by the $\varepsilon$-term. However, for $\varepsilon>0$ small enough, this does not affect our arguments. 
\end{proof}
The realisation we have given is robust with respect to perturbations that leave the coordinate planes invariant, since all connections are of saddle-sink type within these planes.

Note also that in several places different orientations for the parabolas may be chosen -- we give one example here of a realisation for the DNN graph, but ours is certainly not the only one. Indeed, some of the parabolas could even be chosen as vertical lines, which would further reduce the degree of the vector field. We stick with parabolas only in order to avoid intersections of nullclines in the planes outside the axes, even when $n$ increases, thus confining nodes and equilibria to the horizontal axis, as much as possible.

Another realisation of the DNN graph for $n=3$ occurs in the two-player replicator dynamics setup of the Rock-Paper-Scissors game, which is discussed e.g.\ in \cite{GdS-C2020}.

\begin{figure}
\flushright
    \subfigure{
    \begin{tikzpicture}[xscale=2, yscale=2]
  \put (310,50) {$P_{01}$}
  \put (-80,40) {$\begin{cases}[\xi_2 \to \xi_1]\\ [\xi_3 \to \xi_1] \end{cases}$}
  \draw[thick, ->] (0, 0) -- (6, 0) node[right] {$x$};
  \draw[thick, ->] (0, 0) -- (0, 2) node[above] {$y_1$};
  \draw[domain=0:1, smooth, very thick, variable=\y, dashed,blue]  plot ({-\y*\y+1}, {\y});
  \draw[domain=2:3, smooth, very thick, variable=\y, dashed,blue]  plot ({\y},{sqrt(-(\y-(3-1/2))^2+1/4)});
  \draw[domain=4:5, smooth, very thick, variable=\y, dashed,blue]  plot ({\y},{sqrt(-(\y-(5-1/2))^2+1/4)});
  \draw[domain=0:2, smooth, very thick, variable=\y, dotted,red]  plot ({\y*\y+3/2}, {\y});
    \draw[very thick, -latex] (0.5, 0.9) -- (0.5, 0.5);  
    \draw[very thick, -latex] (2.5, 0.3) -- (2.5, 0.7);  
    \draw[very thick, -latex] (4.5, 0.3) -- (4.5, 0.7);  
    \draw[very thick, -latex] (2.4, 0.85) -- (2, 0.85);  
    \draw[very thick, -latex] (3.9, 1.47) -- (3.5, 1.47);  
  \filldraw[black] (1,0) circle (1.5pt) node[below] {$\xi_1=1$};
  \draw[very thick] (1.5, 0.07) -- (1.5, -0.07) node[below] {$\frac{3}{2}$};
  \filldraw[black] (2,0) circle (1pt) node[below] {$2$};
  \filldraw[black] (3,0) circle (1.5pt) node[below] {$\xi_2=3$};
  \filldraw[black] (4,0) circle (1pt) node[below] {$4$};
  \filldraw[black] (5,0) circle (1.5pt) node[below] {$\xi_3=5$};
  \end{tikzpicture}
    }
    \subfigure{
    \begin{tikzpicture}[xscale=2, yscale=2]
  \put (310,50) {$P_{02}$}
  \put (-80,40) {$\begin{cases}[\xi_1 \to \xi_2]\\ [\xi_3 \to \xi_2] \end{cases}$}
  \draw[thick, ->] (0, 0) -- (6, 0) node[right] {$x$};
  \draw[thick, ->] (0, 0) -- (0, 2) node[above] {$y_2$};
  \draw[domain=0:1.5, smooth, very thick, variable=\y, dashed,blue]  plot ({-\y*\y+3}, {\y});
  \draw[domain=1:2, smooth, very thick, variable=\y, dashed,blue]  plot ({\y},{sqrt(-(\y-(2-1/2))^2+1/4)});
  \draw[domain=4:5, smooth, very thick, variable=\y, dashed,blue]  plot ({\y},{sqrt(-(\y-(5-1/2))^2+1/4)});
  \draw[domain=0:1.4, smooth, very thick, variable=\y, dotted,red]  plot ({\y*\y+7/2}, {\y});
  \draw[domain=0:1.5, smooth, very thick, variable=\y, dotted,red]  plot ({-\y*\y+5/2}, {\y});
    \draw[very thick, -latex] (1.5, 0.3) -- (1.5, 0.7);  
    \draw[very thick, -latex] (2.5, 0.9) -- (2.5, 0.5);  
    \draw[very thick, -latex] (4.5, 0.3) -- (4.5, 0.7);  
    \draw[very thick, -latex] (0.9, 1.2) -- (1.3, 1.2);  
    \draw[very thick, -latex] (5.1, 1.2) -- (4.7, 1.2);  
  \filldraw[black] (1,0) circle (1.5pt) node[below] {$\xi_1=1$};
  \filldraw[black] (2,0) circle (1pt) node[below] {$2$};
  \draw[very thick] (2.5, 0.07) -- (2.5, -0.07) node[below] {$\frac{5}{2}$};  
  \filldraw[black] (3,0) circle (1.5pt) node[below] {$\xi_2=3$};
  \draw[very thick] (3.5, 0.07) -- (3.5, -0.07) node[below] {$\frac{7}{2}$};
  \filldraw[black] (4,0) circle (1pt) node[below] {$4$};
  \filldraw[black] (5,0) circle (1.5pt) node[below] {$\xi_3=5$};
  \end{tikzpicture}
    }
    \subfigure{
    \begin{tikzpicture}[xscale=2, yscale=2]
  \put (310,50) {$P_{03}$}
  \put (-80,40) {$\begin{cases}[\xi_1 \to \xi_3]\\ [\xi_2 \to \xi_3] \end{cases}$}
  \draw[thick, ->] (0, 0) -- (6, 0) node[right] {$x$};
  \draw[thick, ->] (0, 0) -- (0, 2) node[above] {$y_3$};
  \draw[domain=0:2, smooth, very thick, variable=\y, dashed,blue]  plot ({-\y*\y+5}, {\y});
  \draw[domain=1:2, smooth, very thick, variable=\y, dashed,blue]  plot ({\y},{sqrt(-(\y-(2-1/2))^2+1/4)});
  \draw[domain=3:4, smooth, very thick, variable=\y, dashed,blue]  plot ({\y},{sqrt(-(\y-(4-1/2))^2+1/4)});
  \draw[domain=0:2, smooth, very thick, variable=\y, dotted,red]  plot ({-\y*\y+9/2}, {\y});
    \draw[very thick, -latex] (1.5, 0.3) -- (1.5, 0.7);  
    \draw[very thick, -latex] (3.5, 0.3) -- (3.5, 0.7);  
    \draw[very thick, -latex] (4.5, 0.9) -- (4.5, 0.5);  
    \draw[very thick, -latex] (2, 1.5) -- (2.4, 1.5);  
    \draw[very thick, -latex] (3.65, 0.8) -- (4.05, 0.8);  
  \filldraw[black] (1,0) circle (1.5pt) node[below] {$\xi_1=1$};
  \filldraw[black] (2,0) circle (1pt) node[below] {$2$};
  \filldraw[black] (3,0) circle (1.5pt) node[below] {$\xi_2=3$};
  \filldraw[black] (4,0) circle (1pt) node[below] {$4$};
  \draw[very thick] (4.5, 0.07) -- (4.5, -0.07) node[below] {$\frac{9}{2}$};
  \filldraw[black] (5,0) circle (1.5pt) node[below] {$\xi_3=5$};
  \end{tikzpicture}
    }
\caption{Nullclines off the axes for $x$ (blue/dashed) and $y_j$ (red/dotted) in the planes $P_{0j}$ for $n=3$. The arrows indicate the direction of the vector field across the nullclines.  In $P_{0j}$, blue parabolas open to the left and intersect the horizontal axis at $\xi_j$, while each blue semi-circle intersects the $x$-axis at $\xi_k$ for $k \neq j$ and at another equilibrium which is not a node. Red parabolas do not intersect the $x$-nullclines when $\varepsilon > 0$ and thus create no additional equilibria. In particular, their intersections with the $x$-axis are not equilibria. In $P_{02}$, two red parabolas are needed since the connections to $\xi_2$ come from different directions.}
\label{fig:n=3}
\end{figure}
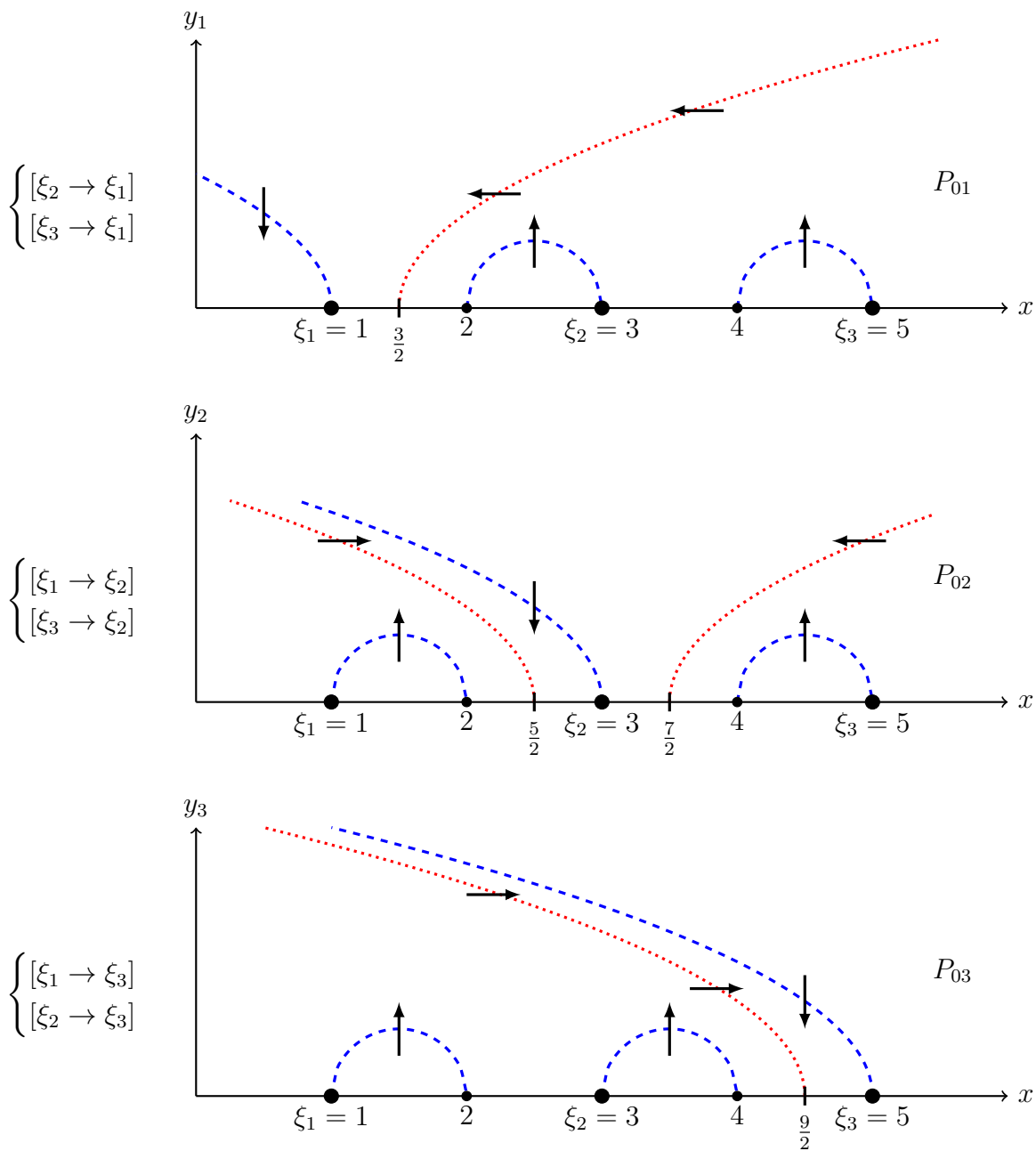

\subsection{A DNN network with four nodes}
\label{subsec:4}

We now introduce a new space dimension and move on to the case $n=4$. The system is set up analogously and we focus on what is new compared to the previous subsection: there are now nodes in a given plane which are not connected, so the nullcline structure must ensure that no undesired connections exist. This creates additional equilibria outside the $x$-axis.
\begin{proposition}
\label{prop_4}
For $\varepsilon>0$ sufficiently small, system~(\ref{sys4}) realises the DNN graph with $n=4$ as a heteroclinic network in $\R^5_{y \geq 0}$ between the nodes $\xi_k:=(2k-1,0,0,0,0)$, where $k=1,\ldots,4$.
\begin{align}
\label{sys4}
\begin{cases}
\dot{x} &=-\varepsilon \prod\limits_{k=1}^7 (x-k)\\
&+ y_1\Big(-y_1^2-x+\xi_1\Big)\Big(y_1^2+(x-(\xi_2-\frac{1}{2}))^2-\frac{1}{4}\Big)\Big(y_1^2+(x-(\xi_3-\frac{1}{2}))^2-\frac{1}{4}\Big)\\
& \quad \times \Big(y_1^2+(x-(\xi_4-\frac{1}{2}))^2-\frac{1}{4}\Big)\\
&+ y_2\Big(-y_2^2-x+\xi_2\Big)\Big(y_2^2+(x-(\xi_1+\frac{1}{2}))^2-\frac{1}{4}\Big)\Big(y_2^2+(x-(\xi_3-\frac{1}{2}))^2-\frac{1}{4}\Big)\\
& \quad \times \Big(y_2^2+(x-(\xi_4-\frac{1}{2}))^2-\frac{1}{4}\Big)\\
&+ y_3\Big(-y_3^2-x+\xi_3\Big)\Big(y_3^2+(x-(\xi_1+\frac{1}{2}))^2-\frac{1}{4}\Big)\Big(y_3^2+(x-(\xi_2+\frac{1}{2}))^2-\frac{1}{4}\Big)\\
& \quad \times \Big(y_3^2+(x-(\xi_4-\frac{1}{2}))^2-\frac{1}{4}\Big)\\
&+ y_4\Big(-y_4^2-x+\xi_4\Big)\Big(y_4^2+(x-(\xi_1+\frac{1}{2}))^2-\frac{1}{4}\Big)\Big(y_4^2+(x-(\xi_2+\frac{1}{2}))^2-\frac{1}{4}\Big)\\
& \quad \times \Big(y_4^2+(x-(\xi_3+\frac{1}{2}))^2-\frac{1}{4}\Big)\\
\dot{y}_1 &= -y_1\Big(y_1^2-x+(\xi_1+\frac{1}{2})\Big)\Big(4y_1^2+(x-(\xi_2-\frac{1}{2}))^2-\frac{1}{2}\Big)\\
\dot{y}_2 &= y_2\Big(-y_2^2-x+(\xi_2-\frac{1}{2})\Big)\Big(y_2^2-x+(\xi_2+\frac{1}{2})\Big)\Big(4y_2^2+(x-(\xi_3-\frac{1}{2}))^2-\frac{1}{2}\Big)\\ 
\dot{y}_3 &= y_3\Big(-y_3^2-x+(\xi_3-\frac{1}{2})\Big)\Big(-y_3^2-x+(\xi_3+\frac{1}{2})\Big)\Big(4y_3^2+(x-(\xi_4-\frac{1}{2}))^2-\frac{1}{2}\Big)\\
\dot{y}_4 &= y_4\Big(-y_4^2-x+(\xi_4-\frac{1}{2})\Big)\Big(4y_4^2+(x-(\xi_1+\frac{1}{2}))^2-\frac{1}{2}\Big)
\end{cases}
\end{align}
\end{proposition}

\begin{proof}
The reader can verify that the system has the required equilibria and it is possible to check that their local stability properties, derived from the Jacobian matrix at each node, are as needed. The nodes $\xi_k$ are such that $x=1,3,5,7$ and alternate with equilibria such that $x=2,4,6$. Along the $x$-axis only the nodes are stable. In each plane $P_{0j}$ the node $\xi_j$ is a sink, while the nodes connecting to it are saddles. There is a fourth node that does not connect to any node in $P_{0j}$, which is a sink, see Figure~\ref{fig:zoom}. 

As before, we study the nullclines to prove the existence of the desired connections:
\begin{itemize}
	\item $[\xi_3 \to \xi_1]$ and $[\xi_4 \to \xi_1]$ in $P_{01}$
	\item $[\xi_4 \to \xi_2]$ and $[\xi_1 \to \xi_2]$ in $P_{02}$
	\item $[\xi_1 \to \xi_3]$ and $[\xi_2 \to \xi_3]$ in $P_{03}$
	\item $[\xi_2 \to \xi_4]$ and $[\xi_3 \to \xi_4]$ in $P_{04}$.
\end{itemize}
All nullclines in the respective planes are shown in Figure~\ref{fig:n=4}. The existence of connections can be deduced just as in the proof of Proposition~\ref{prop_3}. Since the connection structure is not all-to-all anymore we need to prevent connections that are not prescribed by the graph from actually occuring. Additional $y_j$-nullclines are introduced to this end. We illustrate the situation in $P_{01}$ with a close-up in Figure~\ref{fig:zoom}: here, the unwanted connection $[\xi_2 \to \xi_1]$ is prevented by the red ellipse which corresponds to
$$4y_1^2+\left(x-\left(\xi_2-\frac{1}{2}\right)\right)^2-\frac{1}{2}=0.$$
This $y_1$-nullcline is also crucial for keeping $\xi_2$ locally stable in $P_{01}$. Its intersection with the circle which is an $x$-nullcline creates two unstable equilibria outside the $x$-axis.

It remains to show that $\xi_2$ does not connect to $\xi_1$ outside of $P_{01}$. Note that the unstable manifold $W^u(\xi_2)$ is two-dimensional and contained in the three-dimensional invariant space spanned by $P_{01}$ and $P_{03}$. For the restriction of system~\eqref{sys4} to this space, $\xi_1$ is a hyperbolic equilibrium with a two-dimensional stable manifold that is contained in $P_{01}$. Since this is an invariant subspace, there can be no incoming connection to $\xi_1$ outside of $P_{01}$.

Analogous ellipses are needed in the other planes, and another red parabola is added in $P_{03}$ to control the sign of $\dot{y}_3$. The rest of the proof follows in the same way as in subsection~\ref{subsec:n=3}.
\end{proof}

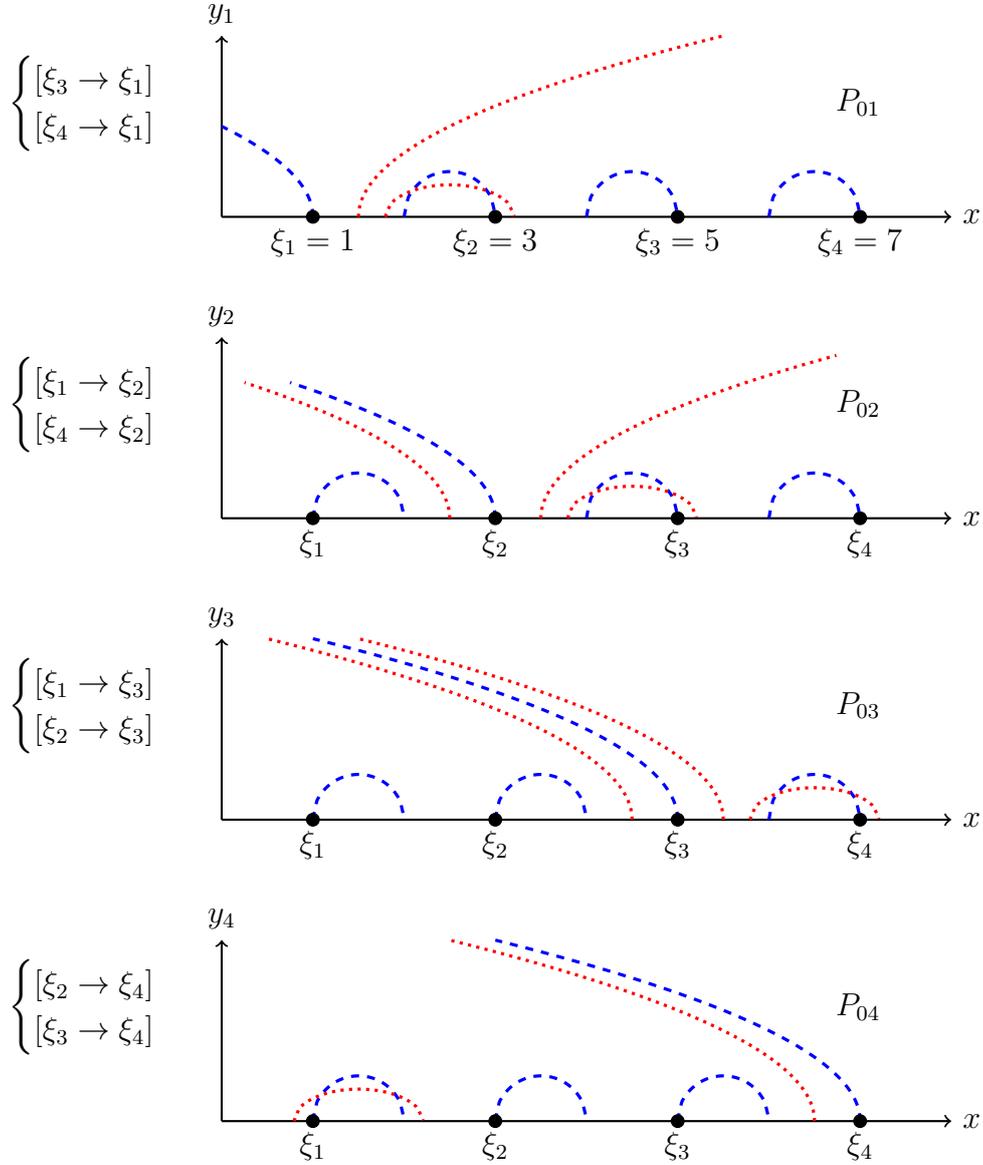
\begin{figure}
\flushright
    \subfigure{
    \begin{tikzpicture}[xscale=1.2, yscale=1.2]
  \put (230,40) {$P_{01}$}
  \put (-80,40) {$\begin{cases}[\xi_3 \to \xi_1]\\ [\xi_4 \to \xi_1] \end{cases}$}
  \draw[thick, ->] (0, 0) -- (8, 0) node[right] {$x$};
  \draw[thick, ->] (0, 0) -- (0, 2) node[above] {$y_1$};
  \draw[domain=0:1, smooth, very thick, variable=\y, dashed, blue]  plot ({-\y*\y+1}, {\y});
  \draw[domain=2:3, smooth, very thick, variable=\y, dashed, blue]  plot ({\y},{sqrt(-(\y-(3-1/2))^2+1/4)});
  \draw[domain=4:5, smooth, very thick, variable=\y, dashed, blue]  plot ({\y},{sqrt(-(\y-(5-1/2))^2+1/4)});
  \draw[domain=6:7, smooth, very thick, variable=\y, dashed, blue]  plot ({\y},{sqrt(-(\y-(7-1/2))^2+1/4)});
  \draw[domain=2.5-1/sqrt(2):2.5+1/sqrt(2), smooth, very thick, variable=\y, dotted, red]  plot ({\y},{sqrt(1/4*(-(\y-(3-1/2))^2+1/2))});
  \draw[domain=0:2, smooth, very thick, variable=\y, dotted, red]  plot ({\y*\y+3/2}, {\y});
  \filldraw[black] (1,0) circle (2pt) node[below] {$\xi_1=1$};
  \filldraw[black] (3,0) circle (2pt) node[below] {$\xi_2=3$};
  \filldraw[black] (5,0) circle (2pt) node[below] {$\xi_3=5$};
  \filldraw[black] (7,0) circle (2pt) node[below] {$\xi_4=7$};
  \end{tikzpicture}
    }
    \subfigure{
    \begin{tikzpicture}[xscale=1.2, yscale=1.2]
  \put (230,40) {$P_{02}$}
  \put (-80,40) {$\begin{cases}[\xi_1 \to \xi_2]\\ [\xi_4 \to \xi_2] \end{cases}$}
  \draw[thick, ->] (0, 0) -- (8, 0) node[right] {$x$};
  \draw[thick, ->] (0, 0) -- (0, 2) node[above] {$y_2$};
  \draw[domain=0:1.5, smooth, very thick, variable=\y, dashed, blue]  plot ({-\y*\y+3}, {\y});
  \draw[domain=1:2, smooth, very thick, variable=\y, dashed, blue]  plot ({\y},{sqrt(-(\y-(2-1/2))^2+1/4)});
  \draw[domain=4:5, smooth, very thick, variable=\y, dashed, blue]  plot ({\y},{sqrt(-(\y-(5-1/2))^2+1/4)});
  \draw[domain=6:7, smooth, very thick, variable=\y, dashed, blue]  plot ({\y},{sqrt(-(\y-(7-1/2))^2+1/4)});
  \draw[domain=4.5-1/sqrt(2):4.5+1/sqrt(2), smooth, very thick, variable=\y, dotted, red]  plot ({\y},{sqrt(1/4*(-(\y-(5-1/2))^2+1/2))});
  \draw[domain=0:1.8, smooth, very thick, variable=\y, dotted, red]  plot ({\y*\y+7/2}, {\y});
  \draw[domain=0:1.5, smooth, very thick, variable=\y, dotted, red]  plot ({-\y*\y+5/2}, {\y});
  \filldraw[black] (1,0) circle (2pt) node[below] {$\xi_1$};
  \filldraw[black] (3,0) circle (2pt) node[below] {$\xi_2$};
  \filldraw[black] (5,0) circle (2pt) node[below] {$\xi_3$};
  \filldraw[black] (7,0) circle (2pt) node[below] {$\xi_4$};
  \end{tikzpicture}
    }
    \subfigure{
     \begin{tikzpicture}[xscale=1.2, yscale=1.2]
  \put (230,40) {$P_{03}$}
  \put (-80,40) {$\begin{cases}[\xi_1 \to \xi_3]\\ [\xi_2 \to \xi_3] \end{cases}$}
  \draw[thick, ->] (0, 0) -- (8, 0) node[right] {$x$};
  \draw[thick, ->] (0, 0) -- (0, 2) node[above] {$y_3$};
  \draw[domain=0:2, smooth, very thick, variable=\y, dashed, blue]  plot ({-\y*\y+5}, {\y});
  \draw[domain=1:2, smooth, very thick, variable=\y, dashed, blue]  plot ({\y},{sqrt(-(\y-(2-1/2))^2+1/4)});
  \draw[domain=3:4, smooth, very thick, variable=\y, dashed, blue]  plot ({\y},{sqrt(-(\y-(4-1/2))^2+1/4)});
  \draw[domain=6:7, smooth, very thick, variable=\y, dashed, blue]  plot ({\y},{sqrt(-(\y-(7-1/2))^2+1/4)});
  \draw[domain=6.5-1/sqrt(2):6.5+1/sqrt(2), smooth, very thick, variable=\y, dotted, red]  plot ({\y},{sqrt(1/4*(-(\y-(7-1/2))^2+1/2))});
  \draw[domain=0:2, smooth, very thick, variable=\y, dotted, red]  plot ({-\y*\y+11/2}, {\y});
  \draw[domain=0:2, smooth, very thick, variable=\y, dotted, red]  plot ({-\y*\y+9/2}, {\y});
  \filldraw[black] (1,0) circle (2pt) node[below] {$\xi_1$};
  \filldraw[black] (3,0) circle (2pt) node[below] {$\xi_2$};
  \filldraw[black] (5,0) circle (2pt) node[below] {$\xi_3$};
  \filldraw[black] (7,0) circle (2pt) node[below] {$\xi_4$};
  \end{tikzpicture}
    }
    \subfigure{
    \begin{tikzpicture}[xscale=1.2, yscale=1.2]
  \put (230,40) {$P_{04}$}
  \put (-80,40) {$\begin{cases}[\xi_2 \to \xi_4]\\ [\xi_3 \to \xi_4] \end{cases}$}
  \draw[thick, ->] (0, 0) -- (8, 0) node[right] {$x$};
  \draw[thick, ->] (0, 0) -- (0, 2) node[above] {$y_4$};
  \draw[domain=0:2, smooth, very thick, variable=\y, dashed, blue]  plot ({-\y*\y+7}, {\y});
  \draw[domain=1:2, smooth, very thick, variable=\y, dashed, blue]  plot ({\y},{sqrt(-(\y-(2-1/2))^2+1/4)});
  \draw[domain=3:4, smooth, very thick, variable=\y, dashed, blue]  plot ({\y},{sqrt(-(\y-(4-1/2))^2+1/4)});
  \draw[domain=5:6, smooth, very thick, variable=\y, dashed, blue]  plot ({\y},{sqrt(-(\y-(6-1/2))^2+1/4)});
  \draw[domain=1.5-1/sqrt(2):1.5+1/sqrt(2), smooth, very thick, variable=\y, dotted, red]  plot ({\y},{sqrt(1/4*(-(\y-(2-1/2))^2+1/2))});
  \draw[domain=0:2, smooth, very thick, variable=\y, dotted, red]  plot ({-\y*\y+7-1/2}, {\y});
  \filldraw[black] (1,0) circle (2pt) node[below] {$\xi_1$};
  \filldraw[black] (3,0) circle (2pt) node[below] {$\xi_2$};
  \filldraw[black] (5,0) circle (2pt) node[below] {$\xi_3$};
  \filldraw[black] (7,0) circle (2pt) node[below] {$\xi_4$};
  \end{tikzpicture}
    }
\caption{Nullclines off the axes for $x$ (blue/dashed) and $y_j$ (red/dotted) in the planes $P_{0j}$ for $n=4$. A red ellipse is introduced in each plane, rendering $\xi_{j+1}$ stable in $P_{0j}$ to prevent unwanted connections between the nodes, see Figure~\ref{fig:zoom} for a close-up. In $P_{03}$ a new red parabola is needed to control the sign of $\dot{y}_3$. Red parabolas always open to the left except in $P_{01}$ and $P_{02}$. The direction of flow across the nullclines can be deduced from the arrows in Figures~\ref{fig:n=3} and \ref{fig:zoom}.}
\label{fig:n=4}
\end{figure}

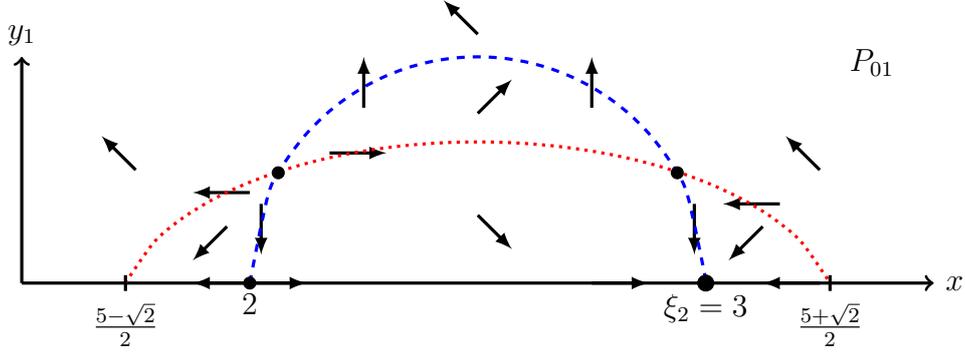
\begin{figure}
\centering
    \begin{tikzpicture}[xscale=1.5, yscale=1.5]
  \put (310,80) {$P_{01}$}
  \draw[very thick, ->] (0, 0) -- (8, 0) node[right] {$x$};
  \draw[very thick, ->] (0, 0) -- (0, 2) node[above] {$y_1$};
  \draw[very thick, -latex] (1, 1) -- (0.7, 1.3);  
  \draw[very thick, -latex] (4, 1.5) -- (4.3, 1.8);
  \draw[very thick, -latex] (4, 0.6) -- (4.3, 0.3);
  \draw[very thick, -latex] (7, 1) -- (6.7, 1.3);
  \draw[very thick, -latex] (6.5, 0.5) -- (6.2, 0.2);
  \draw[very thick, -latex] (4, 2.2) -- (3.7, 2.5);
  \draw[very thick, -latex] (1.8, 0.5) -- (1.5, 0.2);
  \draw[very thick, -latex] (3, 1.55) -- (3, 2);
  \draw[very thick, -latex] (5, 1.55) -- (5, 2);  
  \draw[very thick, -latex] (5.9, 0.7) -- (5.9, 0.25);  
  \draw[very thick, -latex] (2.1, 0.7) -- (2.1, 0.25);  
  \draw[very thick, -latex] (2.7, 1.15) -- (3.2, 1.15);  
  \draw[very thick, -latex] (2, 0.8) -- (1.5, 0.8);  
  \draw[very thick, -latex] (6.65, 0.7) -- (6.15, 0.7);  
  \draw[domain=2:6, smooth, very thick, variable=\y, dashed, blue]  plot ({\y},{sqrt(-(\y-4)^2+4)});
  \draw[domain=4-sqrt(86/9):4+sqrt(86/9), smooth, very thick, variable=\y, dotted, red]  plot ({\y},{sqrt(1/4*(-(\y-4)^2+10))-1/3}); 
  \filldraw[black] (6,0) circle (2pt) node[below] {$\xi_2=3$};
  \filldraw[black] (2,0) circle (1.5pt) node[below] {$2$};
  \filldraw[black] (2.25,0.975) circle (1.5pt) node[below] {};
  \filldraw[black] (5.75,0.975) circle (1.5pt) node[below] {};
  \draw[very thick, -latex] (2, 0) -- (1.5, 0);
  \draw[very thick, -latex] (2, 0) -- (2.5, 0);
  \draw[very thick, -latex] (5, 0) -- (5.5, 0);
  \draw[very thick, -latex] (7, 0) -- (6.5, 0);
  \draw[very thick] (0.91, 0.07) -- (0.91, -0.07) node[below] {$\frac{5-\sqrt{2}}{2}$};
  \draw[very thick] (7.09, 0.07) -- (7.09, -0.07) node[below] {$\frac{5+\sqrt{2}}{2}$};
  \end{tikzpicture}
\caption{Close-up of $P_{01}$ near $\xi_2$. Equilibria that are not nodes are shown as smaller black dots. Arrows indicate the direction of the flow across the nullclines and in the different regions bounded by them. The node $\xi_2$ is a sink. Note that the equilibrium at $x=2$ is a saddle, unlike the other equilibria on the $x$-axis in Figure~\ref{fig:n=4} that are not inside a red ellipse.}
\label{fig:zoom}
\end{figure}

\subsection{A DNN network with five nodes}\label{subsec:n=5}

Moving on to $n=5$ we need another space dimension and implement a new type of ellipse as $y_j$-nullclines to block off several nodes simultaneously.
\begin{proposition}
\label{prop_5}
For $\varepsilon>0$ sufficiently small, system~(\ref{sys5}) realises the DNN graph with $n=5$ as a heteroclinic network in $\R^6_{y \geq 0}$ between the nodes $\xi_k:=(2k-1,0,0,0,0,0)$, where $k=1,\ldots,5$.
\begin{align}
\label{sys5}
\begin{cases}
\dot{x} &=  -\varepsilon \prod\limits_{k=1}^9 (x-k)\\
&+ y_1\Big(-y_1^2-x+\xi_1\Big)\Big(y_1^2+(x-(\xi_2-\frac{1}{2}))^2-\frac{1}{4}\Big)\Big(y_1^2+(x-(\xi_3-\frac{1}{2}))^2-\frac{1}{4}\Big)\\
&\quad \times \Big(y_1^2+(x-(\xi_4-\frac{1}{2}))^2-\frac{1}{4}\Big)\Big(y_1^2+(x-(\xi_5-\frac{1}{2}))^2-\frac{1}{4}\Big)\\
&+ y_2\Big(-y_2^2-x+\xi_2\Big)\Big(y_2^2+(x-(\xi_1+\frac{1}{2}))^2-\frac{1}{4}\Big)\Big(y_2^2+(x-(\xi_3-\frac{1}{2}))^2-\frac{1}{4}\Big)\\
& \quad \times \Big(y_2^2+(x-(\xi_4-\frac{1}{2}))^2-\frac{1}{4}\Big)\Big(y_2^2+(x-(\xi_5-\frac{1}{2}))^2-\frac{1}{4}\Big)\\
&+ y_3\Big(-y_3^2-x+\xi_3\Big)\Big(y_3^2+(x-(\xi_1+\frac{1}{2}))^2-\frac{1}{4}\Big)\Big(y_3^2+(x-(\xi_2+\frac{1}{2}))^2-\frac{1}{4}\Big)\\
& \quad \times \Big(y_3^2+(x-(\xi_4-\frac{1}{2}))^2-\frac{1}{4}\Big)\Big(y_3^2+(x-(\xi_5-\frac{1}{2}))^2-\frac{1}{4}\Big)\\
&+ y_4\Big(-y_4^2-x+\xi_4\Big)\Big(y_4^2+(x-(\xi_1+\frac{1}{2}))^2-\frac{1}{4}\Big)\Big(y_4^2+(x-(\xi_2+\frac{1}{2}))^2-\frac{1}{4}\Big)\\
& \quad \times \Big(y_4^2+(x-(\xi_3+\frac{1}{2}))^2-\frac{1}{4}\Big)\Big(y_4^2+(x-(\xi_5-\frac{1}{2}))^2-\frac{1}{4}\Big)\\
&+ y_5\Big(-y_5^2-x+\xi_5\Big)\Big(y_5^2+(x-(\xi_1+\frac{1}{2}))^2-\frac{1}{4}\Big)\Big(y_5^2+(x-(\xi_2+\frac{1}{2}))^2-\frac{1}{4}\Big)\\
& \quad \times \Big(y_5^2+(x-(\xi_3+\frac{1}{2}))^2-\frac{1}{4}\Big)\Big(y_5^2+(x-(\xi_4+\frac{1}{2}))^2-\frac{1}{4}\Big)\\
\dot{y}_1 &= -y_1\Big(y_1^2-x+(\xi_1+\frac{1}{2})\Big)\Big(16y_1^2+(x-(\xi_2+\frac{1}{2}))^2-3\Big)\\
\dot{y}_2 &= y_2\Big(-y_2^2-x+(\xi_2-\frac{1}{2})\Big)\Big(y_2^2-x+(\xi_2+\frac{1}{2})\Big)\Big(16y_2^2+(x-(\xi_3+\frac{1}{2}))^2-3\Big)\\
\dot{y}_3 &= y_3\Big(-y_3^2-x+(\xi_3-\frac{1}{2})\Big)\Big(-y_3^2-x+(\xi_3+\frac{1}{2})\Big)\Big(16y_3^2+(x-(\xi_4+\frac{1}{2}))^2-3\Big)\\
\dot{y}_4 &= y_4\Big(-y_4^2-x+(\xi_4-\frac{1}{2})\Big)\Big(-y_4^2-x+(\xi_4+\frac{1}{2})\Big)\\
& \quad \times \Big(4y_4^2+(x-(\xi_1+\frac{1}{2}))^2-\frac{1}{2}\Big)\Big(4y_4^2+(x-(\xi_5-\frac{1}{2}))^2-\frac{1}{2}\Big)\\
\dot{y}_5 &= y_5\Big(-y_5^2-x+(\xi_5-\frac{1}{2})\Big)\Big(16y_5^2+(x-(\xi_2-\frac{1}{2}))^2-3\Big)
\end{cases}
\end{align}
\end{proposition}

\begin{proof}
Again, the proof runs along the same lines as that of Propositions~\ref{prop_3} and \ref{prop_4}. 
The nodes $\xi_k$ occur for $x=1,3,5,7,9$ and, along the $x$-axis, alternate with unstable equilibria for $x=2,4,6,8$.
What is new is that there are now several nodes to/from which connections must be inhibited. Following the arguments in the proof of Proposition~\ref{prop_4}, these connections can be prevented by placing ellipses around each of these nodes as in Figure~\ref{fig:zoom}. We do this in $P_{04}$ where the nodes to block are $\xi_1$ and $\xi_5$ and the blocking of connections is implemented exactly as before. In the other planes the same procedure is possible, but since the two nodes in question are directly next to each other, we use one wider red ellipse. E.g., in $P_{01}$ these are $\xi_2$ and $\xi_3$ and they are blocked off (and made locally stable) by the nullcline given through
$$16y_1^2+\left(x-\left(\xi_2+\frac{1}{2}\right)\right)^2-3=0.$$
This allows us to use a polynomial of lower degree by 2 than if we had used one ellipse for each node. Everything else works just as before, the nullclines are shown in Figure~\ref{fig:n=5}.
\end{proof}

\begin{figure}
\flushright
    \subfigure{
    \begin{tikzpicture}[xscale=1.2, yscale=1.2]
  \put (310,45) {$P_{01}$}
  \put (-80,40) {$\begin{cases}[\xi_4 \to \xi_1]\\ [\xi_5 \to \xi_1] \end{cases}$}
  \draw[thick, ->] (0, 0) -- (9.5, 0) node[right] {$x$};
  \draw[thick, ->] (0, 0) -- (0, 2) node[above] {$y_1$};
  \draw[domain=0:1, smooth, very thick, variable=\y, dashed, blue]  plot ({-\y*\y+1}, {\y});
  \draw[domain=2:3, smooth, very thick, variable=\y, dashed, blue]  plot ({\y},{sqrt(-(\y-(3-1/2))^2+1/4)});
  \draw[domain=4:5, smooth, very thick, variable=\y, dashed, blue]  plot ({\y},{sqrt(-(\y-(5-1/2))^2+1/4)});
  \draw[domain=6:7, smooth, very thick, variable=\y, dashed, blue]  plot ({\y},{sqrt(-(\y-(7-1/2))^2+1/4)});
  \draw[domain=8:9, smooth, very thick, variable=\y, dashed, blue]  plot ({\y},{sqrt(-(\y-(9-1/2))^2+1/4)});
  \draw[domain=3.5-sqrt(3):3.5+sqrt(3), smooth, very thick, variable=\y, dotted, red]  plot ({\y},{(sqrt(1/16*(-(\y-(3+1/2))^2+3)))});
  \draw[domain=0:2, smooth, very thick, variable=\y, dotted, red]  plot ({\y*\y+3/2}, {\y});
  \filldraw[black] (1,0) circle (2pt) node[below] {$\xi_1=1$};
  \filldraw[black] (3,0) circle (2pt) node[below] {$\xi_2=3$};
  \filldraw[black] (5,0) circle (2pt) node[below] {$\xi_3=5$};
  \filldraw[black] (7,0) circle (2pt) node[below] {$\xi_4=7$};
  \filldraw[black] (9,0) circle (2pt) node[below] {$\xi_5=9$};
  \end{tikzpicture}
    }
    \subfigure{
    \begin{tikzpicture}[xscale=1.2, yscale=1.2]
  \put (310,45) {$P_{02}$}
  \put (-80,40) {$\begin{cases}[\xi_1 \to \xi_2]\\ [\xi_5 \to \xi_2] \end{cases}$}
  \draw[thick, ->] (0, 0) -- (9.5, 0) node[right] {$x$};
  \draw[thick, ->] (0, 0) -- (0, 2) node[above] {$y_2$};
  \draw[domain=0:1.5, smooth, very thick, variable=\y, dashed, blue]  plot ({-\y*\y+3}, {\y});
  \draw[domain=1:2, smooth, very thick, variable=\y, dashed, blue]  plot ({\y},{sqrt(-(\y-(2-1/2))^2+1/4)});
  \draw[domain=4:5, smooth, very thick, variable=\y, dashed, blue]  plot ({\y},{sqrt(-(\y-(5-1/2))^2+1/4)});
  \draw[domain=6:7, smooth, very thick, variable=\y, dashed, blue]  plot ({\y},{sqrt(-(\y-(7-1/2))^2+1/4)});
  \draw[domain=8:9, smooth, very thick, variable=\y, dashed, blue]  plot ({\y},{sqrt(-(\y-(9-1/2))^2+1/4)});
  \draw[domain=5.5-sqrt(3):5.5+sqrt(3), smooth, very thick, variable=\y, dotted, red]  plot ({\y},{sqrt(1/16*(-(\y-(5+1/2))^2+3))});
  \draw[domain=0:1.8, smooth, very thick, variable=\y, dotted, red]  plot ({\y*\y+7/2}, {\y});
  \draw[domain=0:1.5, smooth, very thick, variable=\y, dotted, red]  plot ({-\y*\y+5/2}, {\y});
  \filldraw[black] (1,0) circle (2pt) node[below] {$\xi_1$};
  \filldraw[black] (3,0) circle (2pt) node[below] {$\xi_2$};
  \filldraw[black] (5,0) circle (2pt) node[below] {$\xi_3$};
  \filldraw[black] (7,0) circle (2pt) node[below] {$\xi_4$};
  \filldraw[black] (9,0) circle (2pt) node[below] {$\xi_5$};
  \end{tikzpicture}
    }
    \subfigure{
    \begin{tikzpicture}[xscale=1.2, yscale=1.2]
  \put (310,45) {$P_{03}$}
  \put (-80,40) {$\begin{cases}[\xi_1 \to \xi_3]\\ [\xi_2 \to \xi_3] \end{cases}$}
  \draw[thick, ->] (0, 0) -- (9.5, 0) node[right] {$x$};
  \draw[thick, ->] (0, 0) -- (0, 2) node[above] {$y_3$};
  \draw[domain=0:2, smooth, very thick, variable=\y, dashed, blue]  plot ({-\y*\y+5}, {\y});
  \draw[domain=1:2, smooth, very thick, variable=\y, dashed, blue]  plot ({\y},{sqrt(-(\y-(2-1/2))^2+1/4)});
  \draw[domain=3:4, smooth, very thick, variable=\y, dashed, blue]  plot ({\y},{sqrt(-(\y-(4-1/2))^2+1/4)});
  \draw[domain=6:7, smooth, very thick, variable=\y, dashed, blue]  plot ({\y},{sqrt(-(\y-(7-1/2))^2+1/4)});
  \draw[domain=8:9, smooth, very thick, variable=\y, dashed, blue]  plot ({\y},{sqrt(-(\y-(9-1/2))^2+1/4)});
  \draw[domain=7.5-sqrt(3):7.5+sqrt(3), smooth, very thick, variable=\y, dotted, red]  plot ({\y},{sqrt(1/16*(-(\y-(7+1/2))^2+3))});
  \draw[domain=0:2, smooth, very thick, variable=\y, dotted, red]  plot ({-\y*\y+11/2}, {\y});
  \draw[domain=0:2, smooth, very thick, variable=\y, dotted, red]  plot ({-\y*\y+9/2}, {\y});
  \filldraw[black] (1,0) circle (2pt) node[below] {$\xi_1$};
  \filldraw[black] (3,0) circle (2pt) node[below] {$\xi_2$};
  \filldraw[black] (5,0) circle (2pt) node[below] {$\xi_3$};
  \filldraw[black] (7,0) circle (2pt) node[below] {$\xi_4$};
  \filldraw[black] (9,0) circle (2pt) node[below] {$\xi_5$};
  \end{tikzpicture}
    }
    \subfigure{
    \begin{tikzpicture}[xscale=1.2, yscale=1.2]
  \put (310,45) {$P_{04}$}
  \put (-80,40) {$\begin{cases}[\xi_2 \to \xi_4]\\ [\xi_3 \to \xi_4] \end{cases}$}
  \draw[thick, ->] (0, 0) -- (9.5, 0) node[right] {$x$};
  \draw[thick, ->] (0, 0) -- (0, 2) node[above] {$y_4$};
  \draw[domain=0:2, smooth, very thick, variable=\y, dashed, blue]  plot ({-\y*\y+7}, {\y});
  \draw[domain=1:2, smooth, very thick, variable=\y, dashed, blue]  plot ({\y},{sqrt(-(\y-(2-1/2))^2+1/4)});
  \draw[domain=3:4, smooth, very thick, variable=\y, dashed, blue]  plot ({\y},{sqrt(-(\y-(4-1/2))^2+1/4)});
  \draw[domain=5:6, smooth, very thick, variable=\y, dashed, blue]  plot ({\y},{sqrt(-(\y-(6-1/2))^2+1/4)});
  \draw[domain=8:9, smooth, very thick, variable=\y, dashed, blue]  plot ({\y},{sqrt(-(\y-(9-1/2))^2+1/4)});
  \draw[domain=1.5-1/sqrt(2):1.5+1/sqrt(2), smooth, very thick, variable=\y, dotted, red]  plot ({\y},{sqrt(1/4*(-(\y-(2-1/2))^2+1/2))});
  \draw[domain=8.5-1/sqrt(2):8.5+1/sqrt(2), smooth, very thick, variable=\y, dotted, red]  plot ({\y},{sqrt(1/4*(-(\y-(9-1/2))^2+1/2))});
  \draw[domain=0:2, smooth, very thick, variable=\y, dotted, red]  plot ({-\y*\y+15/2}, {\y});
  \draw[domain=0:2, smooth, very thick, variable=\y, dotted, red]  plot ({-\y*\y+7-1/2}, {\y});
  \filldraw[black] (1,0) circle (2pt) node[below] {$\xi_1$};
  \filldraw[black] (3,0) circle (2pt) node[below] {$\xi_2$};
  \filldraw[black] (5,0) circle (2pt) node[below] {$\xi_3$};
  \filldraw[black] (7,0) circle (2pt) node[below] {$\xi_4$};
  \filldraw[black] (9,0) circle (2pt) node[below] {$\xi_5$};
  \end{tikzpicture}
    }
        \subfigure{
    \begin{tikzpicture}[xscale=1.2, yscale=1.2]
  \put (310,45) {$P_{05}$}
  \put (-80,40) {$\begin{cases}[\xi_3 \to \xi_5]\\ [\xi_4 \to \xi_5] \end{cases}$}
  \draw[thick, ->] (0, 0) -- (9.5, 0) node[right] {$x$};
  \draw[thick, ->] (0, 0) -- (0, 2) node[above] {$y_5$};
  \draw[domain=0:2, smooth, very thick, variable=\y, dashed, blue]  plot ({-\y*\y+9}, {\y});
  \draw[domain=1:2, smooth, very thick, variable=\y, dashed, blue]  plot ({\y},{sqrt(-(\y-(2-1/2))^2+1/4)});
  \draw[domain=3:4, smooth, very thick, variable=\y, dashed, blue]  plot ({\y},{sqrt(-(\y-(4-1/2))^2+1/4)});
  \draw[domain=5:6, smooth, very thick, variable=\y, dashed, blue]  plot ({\y},{sqrt(-(\y-(6-1/2))^2+1/4)});
  \draw[domain=7:8, smooth, very thick, variable=\y, dashed, blue]  plot ({\y},{sqrt(-(\y-(8-1/2))^2+1/4)});
  \draw[domain=2.5-sqrt(3):2.5+sqrt(3), smooth, very thick, variable=\y, dotted, red]  plot ({\y},{sqrt(1/16*(-(\y-(2+1/2))^2+3))});
  \draw[domain=0:2, smooth, very thick, variable=\y, dotted, red]  plot ({-\y*\y+9-1/2}, {\y});
  \filldraw[black] (1,0) circle (2pt) node[below] {$\xi_1$};
  \filldraw[black] (3,0) circle (2pt) node[below] {$\xi_2$};
  \filldraw[black] (5,0) circle (2pt) node[below] {$\xi_3$};
  \filldraw[black] (7,0) circle (2pt) node[below] {$\xi_4$};
  \filldraw[black] (9,0) circle (2pt) node[below] {$\xi_5$};
  \end{tikzpicture}
    }
\caption{Nullclines off the axes for $x$ (blue/dashed) and $y_j$ (red/dotted) in the planes $P_{0j}$ for $n=5$. In all but $P_{04}$ a wider red ellipse prevents unwanted connections from several nodes at once. The direction of flow across the nullclines can be deduced from the arrows in Figures~\ref{fig:n=3} and \ref{fig:zoom}.}
\label{fig:n=5}
\end{figure}
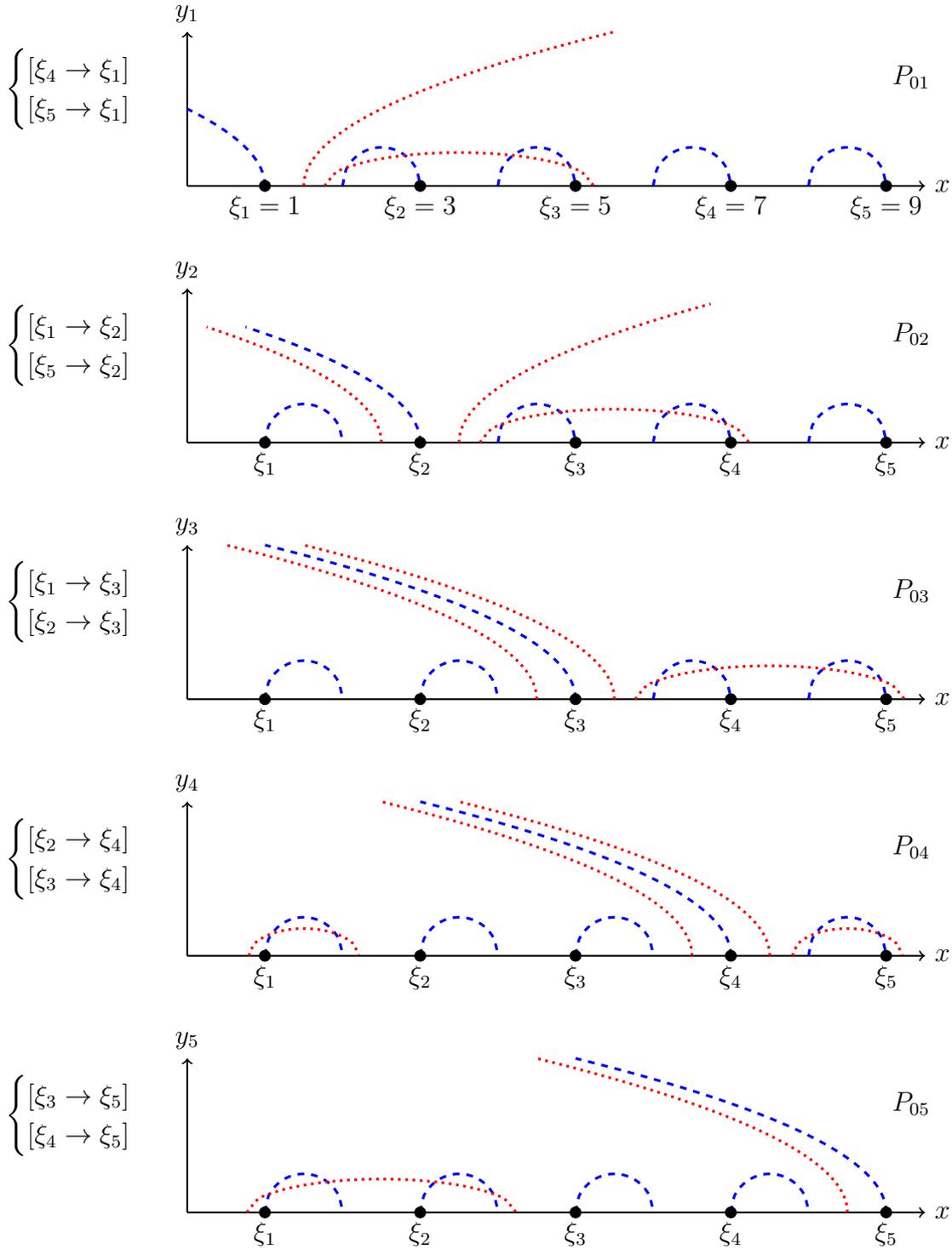

\subsection{A DNN network with six nodes}
\label{subsec:n=6}

For $n=6$ the new and very interesting feature appears that no additional space dimension is needed, i.e.\ we remain in $\R^6_{y \geq 0}$. We achieve this by placing the incoming connections to $\xi_3$ and $\xi_6$ all in the same plane $P_{03}$. Moreover, even wider ellipses are implemented as $y_j$-nullclines in $P_{01}$ and $P_{02}$ to prevent unwanted connections and keep the degree of the vector field down.
\begin{proposition}
\label{prop_6}
For $\varepsilon>0$ sufficiently small, system~(\ref{sys6}) realises the DNN graph with $n=6$ as a heteroclinic network in $\R^6_{y \geq 0}$ between the nodes $\xi_k:=(2k-1,0,0,0,0,0)$ with $k=1,\ldots,6$.
\begin{align}
\label{sys6}
\begin{cases}
\dot{x} &=  -\varepsilon \prod\limits_{k=1}^{11} (x-k)\\
&+ y_1\Big(-y_1^2-x+\xi_1\Big)\Big(y_1^2+(x-(\xi_2-\frac{1}{2}))^2-\frac{1}{4}\Big)\Big(y_1^2+(x-(\xi_3-\frac{1}{2}))^2-\frac{1}{4}\Big)\\
& \quad \times \Big(y_1^2+(x-(\xi_4-\frac{1}{2}))^2-\frac{1}{4}\Big)\Big(y_1^2+(x-(\xi_5-\frac{1}{2}))^2-\frac{1}{4}\Big)\Big(y_1^2+(x-(\xi_6-\frac{1}{2}))^2-\frac{1}{4}\Big)\\
&+ y_2\Big(-y_2^2-x+\xi_2\Big)\Big(y_2^2+(x-(\xi_1+\frac{1}{2}))^2-\frac{1}{4}\Big)\Big(y_2^2+(x-(\xi_3-\frac{1}{2}))^2-\frac{1}{4}\Big)\\
& \quad \times \Big(y_2^2+(x-(\xi_4-\frac{1}{2}))^2-\frac{1}{4}\Big)\Big(y_2^2+(x-(\xi_5-\frac{1}{2}))^2-\frac{1}{4}\Big)\Big(y_2^2+(x-(\xi_6-\frac{1}{2}))^2-\frac{1}{4}\Big)\\
&+ y_3\Big(-y_3^2-x+\xi_3\Big)\Big(y_3^2+(x-(\xi_1+\frac{1}{2}))^2-\frac{1}{4}\Big)\Big(y_3^2+(x-(\xi_2+\frac{1}{2}))^2-\frac{1}{4}\Big)\\
&\quad \times \Big(-y_3^2-x+\xi_6\Big)\Big(y_3^2+(x-(\xi_4+\frac{1}{2}))^2-\frac{1}{4}\Big)\Big(y_3^2+(x-(\xi_5+\frac{1}{2}))^2-\frac{1}{4}\Big)\\
&\quad \times \Big(-y_3^2-x+(\xi_3+1)\Big)\\
&+ y_4\Big(-y_4^2-x+\xi_4\Big)\Big(y_4^2+(x-(\xi_1+\frac{1}{2}))^2-\frac{1}{4}\Big)\Big(y_4^2+(x-(\xi_2+\frac{1}{2}))^2-\frac{1}{4}\Big)\\
& \quad \times \Big(y_4^2+(x-(\xi_3+\frac{1}{2}))^2-\frac{1}{4}\Big)\Big(y_4^2+(x-(\xi_5-\frac{1}{2}))^2-\frac{1}{4}\Big)\Big(y_4^2+(x-(\xi_6-\frac{1}{2}))^2-\frac{1}{4}\Big)\\
&+ y_5\Big(-y_5^2-x+\xi_5\Big)\Big(y_5^2+(x-(\xi_1+\frac{1}{2}))^2-\frac{1}{4}\Big)\Big(y_5^2+(x-(\xi_2+\frac{1}{2}))^2-\frac{1}{4}\Big)\\
& \quad \times \Big(y_5^2+(x-(\xi_3+\frac{1}{2}))^2-\frac{1}{4}\Big)\Big(y_5^2+(x-(\xi_4+\frac{1}{2}))^2-\frac{1}{4}\Big)\Big(y_5^2+(x-(\xi_6-\frac{1}{2}))^2-\frac{1}{4}\Big)\\
\dot{y}_1 &= -y_1\Big(y_1^2-x+(\xi_1+\frac{1}{2})\Big)\Big(64y_1^2+(x-(\xi_3-\frac{1}{2}))^2-7\Big)\\
\dot{y}_2 &= y_2\Big(-y_2^2-x+(\xi_2-\frac{1}{2})\Big)\Big(y_2^2-x+(\xi_2+\frac{1}{2})\Big)\Big(64y_2^2+(x-(\xi_4-\frac{1}{2}))^2-7\Big)\\ 
\dot{y}_3 &= y_3\Big(-y_3^2-x+(\xi_3-\frac{1}{2})\Big)\Big(-y_3^2-x+(\xi_3+\frac{1}{2})\Big)\Big(-y_3^2-x+(\xi_6-\frac{1}{2})\Big)\\
\dot{y}_4 &= y_4\Big(-y_4^2-x+(\xi_4-\frac{1}{2})\Big)\Big(-y_4^2-x+(\xi_4+\frac{1}{2})\Big)\\
& \quad \times \Big(4y_4^2+(x-(\xi_1+\frac{1}{2}))^2-\frac{1}{2}\Big)\Big(16y_4^2+(x-(\xi_5+\frac{1}{2}))^2-3\Big)\\
\dot{y}_5 &= y_5\Big(-y_5^2-x+(\xi_5-\frac{1}{2})\Big)\Big(-y_5^2-x+(\xi_5+\frac{1}{2})\Big)\\
& \quad \times \Big(4y_5^2+(x-(\xi_6-\frac{1}{2}))^2-\frac{1}{2}\Big)\Big(16y_5^2+(x-(\xi_2-\frac{1}{2}))^2-3\Big)
\end{cases}
\end{align}

\end{proposition}

\begin{proof}
All nullclines are displayed in Figure~\ref{fig:n=6}. 
As intended the nodes $\xi_k$ occur for $x=1,3,5,7,9,11$ and, along the $x$-axis, alternate with unstable equilibria $x=2,4,6,8,10$. It is noteworthy that in $P_{03}$ the equilibrium for $x=6$ is created by the intersection of a blue parabola with the $x$-axis, unlike the other equilibria on the $x$-axis which arise from a circle.
The main point to emphasize is in $P_{03}$, where connections $[\xi_1 \to \xi_3]$, $[\xi_2 \to \xi_3]$ and $[\xi_4 \to \xi_6]$, $[\xi_5 \to \xi_6]$ are generated simultaneously, but otherwise in the exact same way as before. Note that there is an additional $x$-nullcline (blue parabola at $\xi_3+1$, opening left) that separates both sets of connections from one another and ensures arrows are pointing in the correct directions in the respective regions. Since all parabolas (blue and red) have the same slope, none of them intersect and there are no new equilibria off the coordinate axes.

In $P_{01}$ there are now three nodes, $\xi_2, \xi_3, \xi_4$, that must not connect to $\xi_1$. A wider ellipse corresponding to the third factor in the equation for $\dot{y}_1$ achieves this while keeping the degree of the vector field lower by 4 than if three smaller ellipses had been used. Analogously in $P_{02}$ and for the nodes $\xi_3, \xi_4, \xi_5$.
\end{proof}

Moving on to $n=7$ (which we do not treat explicitly anymore) what needs to happen is that all incoming connections to $\xi_4$ and $\xi_7$ must lie in $P_{04}$ in order to avoid introducing a new space dimension.

Also in $P_{01}$ and $P_{02}$ a pattern emerges: as $n$ increases, wider and wider ellipses can be used to inhibit all unwanted connections at the same time. In the next section we compute suitable equations for the general case.

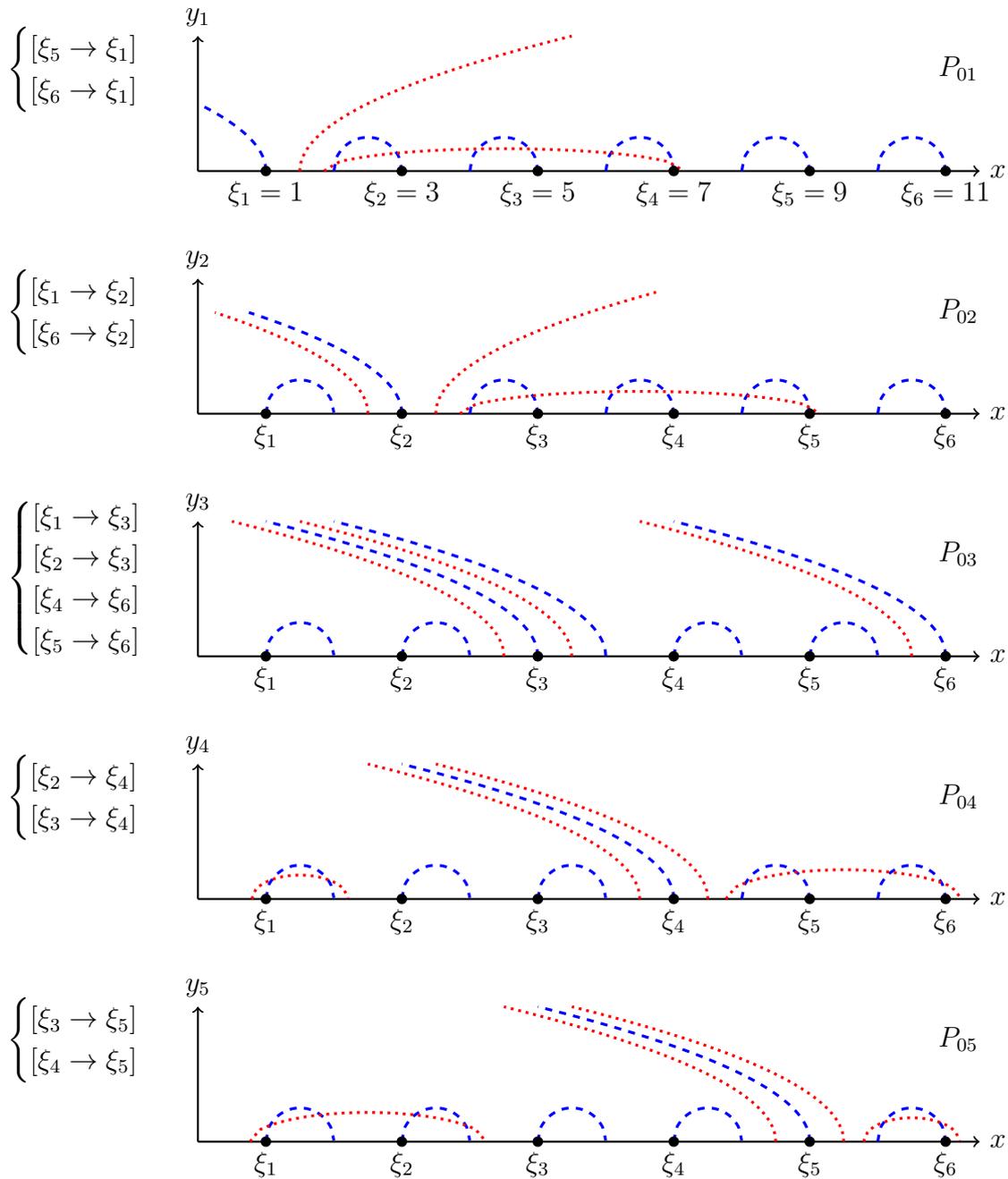
\begin{figure}
\flushright
    \subfigure{
    \begin{tikzpicture}
  \put (310,40) {$P_{01}$}
  \put (-80,40) {$\begin{cases}[\xi_5 \to \xi_1]\\ [\xi_6 \to \xi_1] \end{cases}$}
  \draw[thick, ->] (0, 0) -- (11.5, 0) node[right] {$x$};
  \draw[thick, ->] (0, 0) -- (0, 2) node[above] {$y_1$};
  \draw[domain=0:1, smooth, very thick, variable=\y, dashed, blue]  plot ({-\y*\y+1}, {\y});
  \draw[domain=2:3, smooth, very thick, variable=\y, dashed, blue]  plot ({\y},{sqrt(-(\y-(3-1/2))^2+1/4)});
  \draw[domain=4:5, smooth, very thick, variable=\y, dashed, blue]  plot ({\y},{sqrt(-(\y-(5-1/2))^2+1/4)});
  \draw[domain=6:7, smooth, very thick, variable=\y, dashed, blue]  plot ({\y},{sqrt(-(\y-(7-1/2))^2+1/4)});
  \draw[domain=8:9, smooth, very thick, variable=\y, dashed, blue]  plot ({\y},{sqrt(-(\y-(9-1/2))^2+1/4)});
  \draw[domain=10:11, smooth, very thick, variable=\y, dashed, blue]  plot ({\y},{sqrt(-(\y-(11-1/2))^2+1/4)});
  \draw[domain=4.5-sqrt(7):4.5+sqrt(7), smooth, very thick, variable=\y, dotted, red]  plot ({\y},{(sqrt(1/64*(-(\y-(5-1/2))^2+7)))});
  \draw[domain=0:2, smooth, very thick, variable=\y, dotted, red]  plot ({\y*\y+3/2}, {\y});
  \filldraw[black] (1,0) circle (2pt) node[below] {$\xi_1=1$};
  \filldraw[black] (3,0) circle (2pt) node[below] {$\xi_2=3$};
  \filldraw[black] (5,0) circle (2pt) node[below] {$\xi_3=5$};
  \filldraw[black] (7,0) circle (2pt) node[below] {$\xi_4=7$};
  \filldraw[black] (9,0) circle (2pt) node[below] {$\xi_5=9$};
  \filldraw[black] (11,0) circle (2pt) node[below] {$\xi_6=11$};
  \end{tikzpicture}
    }
    \subfigure{
    \begin{tikzpicture}
  \put (310,40) {$P_{02}$}
  \put (-80,40) {$\begin{cases}[\xi_1 \to \xi_2]\\ [\xi_6 \to \xi_2] \end{cases}$}
  \draw[thick, ->] (0, 0) -- (11.5, 0) node[right] {$x$};
  \draw[thick, ->] (0, 0) -- (0, 2) node[above] {$y_2$};
  \draw[domain=0:1.5, smooth, very thick, variable=\y, dashed, blue]  plot ({-\y*\y+3}, {\y});
  \draw[domain=1:2, smooth, very thick, variable=\y, dashed, blue]  plot ({\y},{sqrt(-(\y-(2-1/2))^2+1/4)});
  \draw[domain=4:5, smooth, very thick, variable=\y, dashed, blue]  plot ({\y},{sqrt(-(\y-(5-1/2))^2+1/4)});
  \draw[domain=6:7, smooth, very thick, variable=\y, dashed, blue]  plot ({\y},{sqrt(-(\y-(7-1/2))^2+1/4)});
  \draw[domain=8:9, smooth, very thick, variable=\y, dashed, blue]  plot ({\y},{sqrt(-(\y-(9-1/2))^2+1/4)});
  \draw[domain=10:11, smooth, very thick, variable=\y, dashed, blue]  plot ({\y},{sqrt(-(\y-(11-1/2))^2+1/4)});
  \draw[domain=6.5-sqrt(7):6.5+sqrt(7), smooth, very thick, variable=\y, dotted, red]  plot ({\y},{sqrt(1/64*(-(\y-(7-1/2))^2+7))});
  \draw[domain=0:1.8, smooth, very thick, variable=\y, dotted, red]  plot ({\y*\y+7/2}, {\y});
  \draw[domain=0:1.5, smooth, very thick, variable=\y, dotted, red]  plot ({-\y*\y+5/2}, {\y});
  \filldraw[black] (1,0) circle (2pt) node[below] {$\xi_1$};
  \filldraw[black] (3,0) circle (2pt) node[below] {$\xi_2$};
  \filldraw[black] (5,0) circle (2pt) node[below] {$\xi_3$};
  \filldraw[black] (7,0) circle (2pt) node[below] {$\xi_4$};
  \filldraw[black] (9,0) circle (2pt) node[below] {$\xi_5$};
  \filldraw[black] (11,0) circle (2pt) node[below] {$\xi_6$};
  \end{tikzpicture}
    }
    \subfigure{
    \begin{tikzpicture}
  \put (310,40) {$P_{03}$}
  \put (-80,30) {$\begin{cases}[\xi_1 \to \xi_3]\\ [\xi_2 \to \xi_3]\\ [\xi_4 \to \xi_6]\\ [\xi_5 \to \xi_6] \end{cases}$}
  \draw[thick, ->] (0, 0) -- (11.5, 0) node[right] {$x$};
  \draw[thick, ->] (0, 0) -- (0, 2) node[above] {$y_3$};
  \draw[domain=0:2, smooth, very thick, variable=\y, dashed, blue]  plot ({-\y*\y+5}, {\y});
  \draw[domain=0:2, smooth, very thick, variable=\y, dashed, blue]  plot ({-\y*\y+6}, {\y});
  \draw[domain=0:2, smooth, very thick, variable=\y, dashed, blue]  plot ({-\y*\y+11}, {\y});
  \draw[domain=1:2, smooth, very thick, variable=\y, dashed, blue]  plot ({\y},{sqrt(-(\y-(2-1/2))^2+1/4)});
  \draw[domain=3:4, smooth, very thick, variable=\y, dashed, blue]  plot ({\y},{sqrt(-(\y-(4-1/2))^2+1/4)});
  \draw[domain=7:8, smooth, very thick, variable=\y, dashed, blue]  plot ({\y},{sqrt(-(\y-(8-1/2))^2+1/4)});
  \draw[domain=9:10, smooth, very thick, variable=\y, dashed, blue]  plot ({\y},{sqrt(-(\y-(10-1/2))^2+1/4)});
  \draw[domain=0:2, smooth, very thick, variable=\y, dotted, red]  plot ({-\y*\y+9/2}, {\y});
  \draw[domain=0:2, smooth, very thick, variable=\y, dotted, red]  plot ({-\y*\y+11/2}, {\y});
  \draw[domain=0:2, smooth, very thick, variable=\y, dotted, red]  plot ({-\y*\y+21/2}, {\y});
  \filldraw[black] (1,0) circle (2pt) node[below] {$\xi_1$};
  \filldraw[black] (3,0) circle (2pt) node[below] {$\xi_2$};
  \filldraw[black] (5,0) circle (2pt) node[below] {$\xi_3$};
  \filldraw[black] (7,0) circle (2pt) node[below] {$\xi_4$};
  \filldraw[black] (9,0) circle (2pt) node[below] {$\xi_5$};
  \filldraw[black] (11,0) circle (2pt) node[below] {$\xi_6$};
  \end{tikzpicture}
    }
    \subfigure{
    \begin{tikzpicture}
  \put (310,40) {$P_{04}$}
  \put (-80,40) {$\begin{cases}[\xi_2 \to \xi_4]\\ [\xi_3 \to \xi_4] \end{cases}$}
  \draw[thick, ->] (0, 0) -- (11.5, 0) node[right] {$x$};
  \draw[thick, ->] (0, 0) -- (0, 2) node[above] {$y_4$};
  \draw[domain=0:2, smooth, very thick, variable=\y, dashed, blue]  plot ({-\y*\y+7}, {\y});
  \draw[domain=1:2, smooth, very thick, variable=\y, dashed, blue]  plot ({\y},{sqrt(-(\y-(2-1/2))^2+1/4)});
  \draw[domain=3:4, smooth, very thick, variable=\y, dashed, blue]  plot ({\y},{sqrt(-(\y-(4-1/2))^2+1/4)});
  \draw[domain=5:6, smooth, very thick, variable=\y, dashed, blue]  plot ({\y},{sqrt(-(\y-(6-1/2))^2+1/4)});
  \draw[domain=8:9, smooth, very thick, variable=\y, dashed, blue]  plot ({\y},{sqrt(-(\y-(9-1/2))^2+1/4)});
  \draw[domain=10:11, smooth, very thick, variable=\y, dashed, blue]  plot ({\y},{sqrt(-(\y-(11-1/2))^2+1/4)});
  \draw[domain=1.5-1/sqrt(2):1.5+1/sqrt(2), smooth, very thick, variable=\y, dotted, red]  plot ({\y},{sqrt(1/4*(-(\y-(2-1/2))^2+1/2))});
  \draw[domain=9.5-sqrt(3):9.5+sqrt(3), smooth, very thick, variable=\y, dotted, red]  plot ({\y},{sqrt(1/16*(-(\y-(9+1/2))^2+3))});
  \draw[domain=0:2, smooth, very thick, variable=\y, dotted, red]  plot ({-\y*\y+15/2}, {\y});
  \draw[domain=0:2, smooth, very thick, variable=\y, dotted, red]  plot ({-\y*\y+7-1/2}, {\y});
  \filldraw[black] (1,0) circle (2pt) node[below] {$\xi_1$};
  \filldraw[black] (3,0) circle (2pt) node[below] {$\xi_2$};
  \filldraw[black] (5,0) circle (2pt) node[below] {$\xi_3$};
  \filldraw[black] (7,0) circle (2pt) node[below] {$\xi_4$};
  \filldraw[black] (9,0) circle (2pt) node[below] {$\xi_5$};
  \filldraw[black] (11,0) circle (2pt) node[below] {$\xi_6$};
  \end{tikzpicture}
    }
        \subfigure{
    \begin{tikzpicture}
  \put (310,40) {$P_{05}$}
  \put (-80,40) {$\begin{cases}[\xi_3 \to \xi_5]\\ [\xi_4 \to \xi_5] \end{cases}$}
  \draw[thick, ->] (0, 0) -- (11.5, 0) node[right] {$x$};
  \draw[thick, ->] (0, 0) -- (0, 2) node[above] {$y_5$};
  \draw[domain=0:2, smooth, very thick, variable=\y, dashed, blue]  plot ({-\y*\y+9}, {\y});
  \draw[domain=1:2, smooth, very thick, variable=\y, dashed, blue]  plot ({\y},{sqrt(-(\y-(2-1/2))^2+1/4)});
  \draw[domain=3:4, smooth, very thick, variable=\y, dashed, blue]  plot ({\y},{sqrt(-(\y-(4-1/2))^2+1/4)});
  \draw[domain=5:6, smooth, very thick, variable=\y, dashed, blue]  plot ({\y},{sqrt(-(\y-(6-1/2))^2+1/4)});
  \draw[domain=7:8, smooth, very thick, variable=\y, dashed, blue]  plot ({\y},{sqrt(-(\y-(8-1/2))^2+1/4)});
  \draw[domain=10:11, smooth, very thick, variable=\y, dashed, blue]  plot ({\y},{sqrt(-(\y-(11-1/2))^2+1/4)});
  \draw[domain=2.5-sqrt(3):2.5+sqrt(3), smooth, very thick, variable=\y, dotted, red]  plot ({\y},{sqrt(1/16*(-(\y-(2+1/2))^2+3))});
    \draw[domain=10.5-1/sqrt(2):10.5+1/sqrt(2), smooth, very thick, variable=\y, dotted, red]  plot ({\y},{sqrt(1/4*(-(\y-(11-1/2))^2+1/2))});
  \draw[domain=0:2, smooth, very thick, variable=\y, dotted, red]  plot ({-\y*\y+19/2}, {\y});
  \draw[domain=0:2, smooth, very thick, variable=\y, dotted, red]  plot ({-\y*\y+9-1/2}, {\y});
  \filldraw[black] (1,0) circle (2pt) node[below] {$\xi_1$};
  \filldraw[black] (3,0) circle (2pt) node[below] {$\xi_2$};
  \filldraw[black] (5,0) circle (2pt) node[below] {$\xi_3$};
  \filldraw[black] (7,0) circle (2pt) node[below] {$\xi_4$};
  \filldraw[black] (9,0) circle (2pt) node[below] {$\xi_5$};
  \filldraw[black] (11,0) circle (2pt) node[below] {$\xi_6$};
  \end{tikzpicture}
    }
\caption{Nullclines off the axes for $x$ (blue/dashed) and $y_j$ (red/dotted) in the planes $P_{0j}$ for $n=6$. In $P_{03}$ a new blue parabola separates the regions for incoming connections at two different nodes. The last new feature appears in $P_{03}$ where a blue parabola intersects the $x$-axis at $x=6$ instead of at a node. This parabola separates the incoming connections to $\xi_3$ from the incoming connections to $\xi_6$. Notice that in $P_{01}$ and $P_{02}$ there are ellipses over three nodes, blocking connections. Again, the direction of flow across the nullclines can be deduced from the arrows in Figures~\ref{fig:n=3} and \ref{fig:zoom}.}
\label{fig:n=6}
\end{figure}

\section{DNN networks for arbitrary $n \in \N$}
\label{sec_n}

In the previous section we have covered the cases $n=3,4,5,6$ explicitly. What we learned there can be generalized to realise any DNN graph as a heteroclinic network in $\R^6_{y \geq 0}$, though the notation is at times cumbersome.

The main idea can be summarized as follows:
\begin{itemize}
	\item The planes $P_{01}$ and $P_{02}$ are used for the incoming connections to $\xi_1$ and $\xi_2$, respectively. The nullcline structure is simple, we create wider and wider red ellipses ($y_j$-nullclines for $j=1,2$) to ensure local stability of all nodes and prevent unwanted connections. Moreover, these are the only planes with connections from right to left (two in $P_{01}$ and one in $P_{02}$).
	\item As $n$ increases more and more sets of incoming connections are built into $P_{03}$, $P_{04}$ and $P_{05}$. Thus, in these planes no more than three nodes remain unconnected and have to be blocked by red ellipses. More precisely: the incoming connections to $\xi_j, \xi_{j+3},\ldots,\xi_{j+3l}$ (with $j+3l \leq n$ and $j=3,4,5$) are all in $P_{0j}$, they all go from left to right. As an example, the nullclines in $P_{04}$ for $n=7$ are schematically depicted in Figure~\ref{fig:n=7}. Note the similarity to $P_{03}$ for $n=6$ in Figure~\ref{fig:n=6}.
\end{itemize}

\begin{figure}
\centering
    \begin{tikzpicture}
  \put (355,45) {$P_{04}$}
  \draw[thick, ->] (0, 0) -- (13.5, 0) node[right] {$x$};
  \draw[thick, ->] (0, 0) -- (0, 2) node[above] {$y_4$};
  \draw[domain=0:2, smooth, very thick, variable=\y, dashed, blue]  plot ({-\y*\y+7}, {\y});
  \draw[domain=1:2, smooth, very thick, variable=\y, dashed, blue]  plot ({\y},{sqrt(-(\y-(2-1/2))^2+1/4)});
  \draw[domain=3:4, smooth, very thick, variable=\y, dashed, blue]  plot ({\y},{sqrt(-(\y-(4-1/2))^2+1/4)});
  \draw[domain=5:6, smooth, very thick, variable=\y, dashed, blue]  plot ({\y},{sqrt(-(\y-(6-1/2))^2+1/4)});
  \draw[domain=9:10, smooth, very thick, variable=\y, dashed, blue]  plot ({\y},{sqrt(-(\y-(9+1/2))^2+1/4)});
  \draw[domain=11:12, smooth, very thick, variable=\y, dashed, blue]  plot ({\y},{sqrt(-(\y-(11+1/2))^2+1/4)});
  \draw[domain=1.5-1/sqrt(2):1.5+1/sqrt(2), smooth, very thick, variable=\y, dotted, red]  plot ({\y},{sqrt(1/4*(-(\y-(2-1/2))^2+1/2))});
  \draw[domain=0:2, smooth, very thick, variable=\y, dashed, blue]  plot ({-\y*\y+13}, {\y});
  \draw[domain=0:2, smooth, very thick, variable=\y, dashed, blue]  plot ({-\y*\y+8}, {\y});
  \draw[domain=1.5-1/sqrt(2):1.5+1/sqrt(2), smooth, very thick, variable=\y, dotted, red]  plot ({\y},{sqrt(1/4*(-(\y-(2-1/2))^2+1/2))});
  \draw[domain=0:2, smooth, very thick, variable=\y, dotted, red]  plot ({-\y*\y+15/2}, {\y});
  \draw[domain=0:2, smooth, very thick, variable=\y, dotted, red]  plot ({-\y*\y+7-1/2}, {\y});
  \draw[domain=0:2, smooth, very thick, variable=\y, dotted, red]  plot ({-\y*\y+13-1/2}, {\y});
  \filldraw[black] (1,0) circle (2pt) node[below] {$\xi_1$};
  \filldraw[black] (3,0) circle (2pt) node[below] {$\xi_2$};
  \filldraw[black] (5,0) circle (2pt) node[below] {$\xi_3$};
  \filldraw[black] (7,0) circle (2pt) node[below] {$\xi_4$};
  \filldraw[black] (9,0) circle (2pt) node[below] {$\xi_5$};
  \filldraw[black] (11,0) circle (2pt) node[below] {$\xi_6$};
  \filldraw[black] (13,0) circle (2pt) node[below] {$\xi_7$};
  \end{tikzpicture}
\caption{Nullclines off the axes for $x$ (blue/dashed) and $y_4$ (red/dotted) in the plane $P_{04}$ for $n=7$. There are connections $[\xi_2 \to \xi_4]$, $[\xi_3 \to \xi_4]$ and $[\xi_5 \to \xi_7]$, $[\xi_6 \to \xi_7]$.}
\label{fig:n=7}
\end{figure}
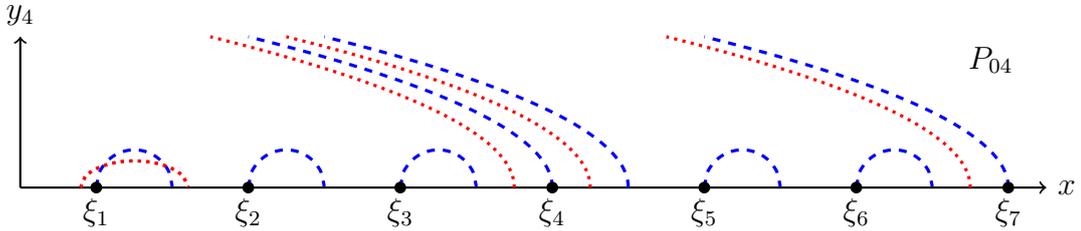

\begin{theorem}
\label{thm1}
For $\varepsilon>0$ sufficiently small and suitable choices of polynomial functions $f_i, g_i: \R^2 \to \R$, $i=1,\ldots ,5$, system~(\ref{sysn}) realises the DNN graph with $n \in \N$ as a heteroclinic network in $\R^6_{y \geq 0}$ between the nodes $\xi_k:=(2k-1,0,0,0,0,0)$, where $k=1,\ldots,n$.
\begin{align}
\label{sysn}
\begin{cases}
\dot{x} &= -\varepsilon \prod\limits_{k=1}^{2n-1} (x-k) + \sum\limits_{j=1}^5 y_jf_j(x,y_j)\\
\dot{y}_j &= y_jg_j(x,y_j)
\end{cases}
\end{align}

\end{theorem}

\begin{proof}
As before, the existence of the desired connections comes down to the nullclines of system~(\ref{sysn}) in $P_{0j}$. From the previous constructions it follows that for $j=1,\ldots,5$ we must define $f_j, g_j$ with suitable (semi)circles, ellipses and parabolas such that (i) the correct equilibria are generated, (ii) the local stability at the nodes checks out and (iii) desired connections are created while unwanted ones are inhibited.

We start with the $x$-nullclines, and thus with $f_j$. All we need here is the standard (semi)circles and parabolas we have already used before. As usual, identifying $\xi_k$ with its non-zero coordinate we define
\begin{align*}
P_+(x,y_j,\xi_k) &:= \ \ y_j^2-x+\xi_k\\
P_-(x,y_j,\xi_k) &:= -y_j^2-x+\xi_k\\
E_+(x,y_j,\xi_k) &:= y_j^2+\left(x-\left(\xi_k+\frac{1}{2}\right)\right)^2-\frac{1}{4}\\
E_-(x,y_j,\xi_k) &:= y_j^2+\left(x-\left(\xi_k-\frac{1}{2}\right)\right)^2-\frac{1}{4},
\end{align*}
where $P_+$ and $P_-$ are parabolas intersecting the $x$-axis at $\xi_k$. The $+$ sign indicates that the parabola opens to the right; the $-$ sign that it opens to the left. Expressions $E_+$ and $E_-$ are ellipses, more precisely circles, centered at $\xi_k + \frac{1}{2}$ and $\xi_k - \frac{1}{2}$, respectively, accounting for the choice of $+$ and $-$ as an index. These intersect the $x$-axis at $\xi_k$ and at the equilibrium either to the right of $\xi_k$, in the case of $E_+$, or to the left of $\xi_k$, in the case of $E_-$, giving further meaning to the $+$ and $-$ signs.

For $j=1,\ldots,5$ and index sets defined in~\eqref{def:index1}--\eqref{def:index3} define:
\begin{align*}
f_j(x,y_j)&:=\tilde{f}_j(x,y_j) \prod_{k \in I_{j,+,n}} E_+(x,y_j,\xi_k)  \prod_{k \in I_{j,-,n}} E_-(x,y_j,\xi_k) \prod_{k \not\in I_{j,\pm,n}}P_-(x,y_j,\xi_k).
\end{align*}
We have $\tilde{f}_1 = \tilde{f}_2 \equiv 1$ and deal with the other $\tilde{f}_j$ further below. The first/second product creates all equilibria to the left/right of the largest node with an incoming connection in $P_{0j}$, by placing an appropriate circle as an $x$-nullcline. The third product creates the nodes with incoming connections in $P_{0j}$.

We now discuss the index sets $I_{j, \pm, n}$. In accordance with the definitions of the parabolas $P_\pm$ and ellipses/circles $E_\pm$ above, the idea is the following: The sets $I_{j,\pm,n}$ contain all those indices $k$ for which a blue semi-circle $x$-nullcline is attached to $\xi_k$ in $P_{0j}$. For those in $I_{j,+,n}$ the center of the semi-circle is to the right of $\xi_k$ at $\xi_k + \frac{1}{2}$, for those in $I_{j,-,n}$ it is to the left of $\xi_k$ at $\xi_k - \frac{1}{2}$. A list of the sets for $n\leq 8$ is given in Remark~\ref{rem:index-sets}. It can be compared with Figures~\ref{fig:n=3},\ref{fig:n=4},\ref{fig:n=5} and \ref{fig:n=6} to verify the positions of the (blue) $x$-nullclines. For any $n$ we have
\begin{align}
\label{def:index1}
&I_{1,+,n}:= \{ \ \}, \quad &&I_{1,-,n}:=\{2,3, \ldots n\}\\
&I_{2,+,n}:= \{ 1 \}, \quad &&I_{2,-,n}:=\{3,4, \ldots n\}
\end{align}
since the only incoming connections in $P_{01}$ and $P_{02}$ are to $\xi_1$ and $\xi_2$, respectively, as pointed out above.

The other index sets $I_{j,\pm,n}$ are as follows
\begin{align}
\label{def:index2}
I_{3,+,n} &:=\{1, \ldots ,n \} \setminus \left( \{3k \mid k \in \N \} \cup I_{3,-,n} \right)\\
I_{4,+,n} &:=\{1, \ldots ,n \} \setminus \left( \{3k+1 \mid k \in \N \} \cup I_{4,-,n} \right)\\
I_{5,+,n} &:=\{1, \ldots ,n \} \setminus \left( \{3k+2 \mid k \in \N \} \cup I_{5,-,n} \right)
\end{align}
while:
\begin{align}
&\text{for $n=3k$:} &&I_{3,-,n}:=\{\ \} \quad &&I_{4,-,n}:=\{n-1 ,n \} \quad &&I_{5,-,n}:=\{n \}\\
&\text{for $n=3k+1$:} &&I_{3,-,n}:=\{n \} \quad &&I_{4,-,n}:=\{\ \} \quad &&I_{5,-,n}:=\{n-1 ,n \}\\
&\text{for $n=3k+2$:} &&I_{3,-,n}:=\{n-1, n \} \quad &&I_{4,-,n}:=\{n \} \quad &&I_{5,-,n}:=\{\ \}\label{def:index3}
\end{align}

The remaining $\tilde{f}_j(x,y_j)$ creates the first equilibrium (not a node) to the right of each set of three nodes, e.g.\ $\xi_3 + 1 \in P_{03}$ for $n=6$, see Figure~\ref{fig:n=6}. In general, this can be expressed as
\begin{align*}
\tilde{f}_j(x,y_j)&:= \prod_{k \in I_j} P_-(x,y_j,\xi_k+1)
\end{align*}
with
\begin{align*}
I_{3} &:=\{3,6,9, \ldots ,3m \ | \ 3m<n \}\\
I_{4} &:=\{4,7,10, \ldots ,3m+1 \ | \ 3m+1<n \}\\
I_{5} &:=\{5,8,11, \ldots ,3m+2 \ | \ 3m+2<n \}
\end{align*}
In other words, the product runs over all $k$ in $\{1,2,\ldots,n\} \setminus I_{j,\pm,n}$, except for the biggest element of this set.

This concludes the discussion of the $x$-nullclines in the form of the functions $f_j$. Note that for $n \leq 6$ the only non-trival $\tilde{f}_j$ is for $n=6$:
$$\tilde{f}_3(x,y_3):=-y_3^2-x-(\xi_3+1) = P_-(x,y_3,\xi_3+1).$$

We continue with the functions $g_j$ for $j = 1,\ldots ,5$. Here we need general expressions for wider and wider ellipses covering a certain number $l$ of smaller (blue) ellipses/(semi)circles. Such ellipses occur only in $P_{01}$ and $P_{02}$. 

Let us write
$$E_l(x,y_j,c):=ay_j^2 + (x-c)^2-b=0$$
and discuss the roles of the parameters $a,b,c$ for each $l$: $a$ controls the height of the ellipse (needs to be less than 1, so as not to interfere with other nullclines), $b$ controls the width (needs to be greater than $2l-1$ if $l$ blue ellipses are to be covered), and $c$ controls where the ellipse is centered (this is one half of the minimal width plus 2 (resp.\ 4) in $P_{01}$ (resp.\ $P_{02}$)). This yields $c=l+\frac{3}{2}$ in $P_{01}$ and $c=l+\frac{7}{2}$ in $P_{02}$.

The conditions for the parameters $a,b$ are the same in both planes. It follows that $a>b$ and $b>(l-\frac{1}{2})^2$ are sufficient. This fits our choices for the cases of $l=1,2,3$ in the previous section.

For $n\geq 4$, the functions $g_1$ and $g_2$ can now be set as:

\begin{align*}
g_1(x,y_1)&:=P_+\left(x,y_1,\xi_1+\frac{1}{2}\right)E_{n-3}\left(x,y_1,n-\frac{3}{2}\right);\\
g_2(x,y_2)&:=P_-\left(x,y_2,\xi_2-\frac{1}{2}\right)P_+\left(x,y_2,\xi_2+\frac{1}{2}\right)E_{n-3}\left(x,y_1,n+\frac{1}{2}\right).
\end{align*}
The ellipses $E_{n-3}$ are wide enough to cover the $n-3$ nodes in $P_{01}$ and $P_{02}$ that do not connect to $\xi_1$ and $\xi_2$, respectively. As pointed out in subsections~\ref{subsec:n=5} and \ref{subsec:n=6} using one wider ellipse instead of $n-3$ smaller ones keeps the degree of $g_1$ and $g_2$ constant as the number of nodes increases: the degree of $g_1$ is $4$ and that of $g_2$ is $6$ when using one ellipse. If we were using $n-3$ ellipses the degrees would be, respectively, $2+2(n-3)=2n-4$ and $4+2(n-3)=2n-2$.

For the other $j$ we distinguish the respective subcases $n$ modulo 3. In $P_{0j}$, $j=3,4,5$, we need only $E_1(x,y_j,c)$ or $E_2(x,y_j,c)$ to cover circles either at the start or at the end of the line of nodes. At the beginning we have $E_1(x,y_j,\xi_1+\frac{1}{2})$ or $E_2(x,y_j,\xi_1+\frac{3}{2})$, whereas at the end we use either $E_1(x,y_j,\xi_n-\frac{1}{2})$ or $E_2(x,y_j,\xi_n-\frac{3}{2})$. The function $P_+(x,y_j,c)$, which produces parabolas opening to the right, is not used in these planes.

The use of $P_-(x,y_j,c)$ is described as a function of the number of nodes $n$, modulo 3. For $j=3$ we have:
\begin{itemize}
\item $n=3k$:
\begin{align*}
g_3(x,y_3)&:=P_-\left(x,y_3,\xi_n-\frac{1}{2}\right)\prod\limits_{m=1}^{k-1}P_-\left(x,y_3,\xi_{3m} \pm\frac{1}{2}\right)
\end{align*}
\item $n=3k+1$:
\begin{align*}
g_3(x,y_3)&:=E_1\left(x,y_3,\xi_n-\frac{1}{2}\right) \prod\limits_{m=1}^{k}P_-\left(x,y_3,\xi_{3m} \pm \frac{1}{2}\right)
\end{align*}
\item $n=3k+2$:
\begin{align*}
g_3(x,y_3)&:=E_2\left(x,y_3,\xi_n-\frac{3}{2}\right) \prod\limits_{m=1}^{k}P_-\left(x,y_3,\xi_{3m} \pm \frac{1}{2}\right)
\end{align*}
\end{itemize}

For $j=4$ we have:
\begin{itemize}
\item $n=3k$:
\begin{align*}
g_4(x,y_4)&:= E_1\left(x,y_4,\xi_1+\frac{1}{2}\right) E_2\left(x,y_4,\xi_n-\frac{3}{2}\right) \prod\limits_{m=1}^{k-1}P_-\left(x,y_4,\xi_{3m+1}\pm\frac{1}{2}\right)
\end{align*}
\item $n=3k+1$:
\begin{align*}
g_4(x,y_4)&:=P_-\left(x,y_4,\xi_n-\frac{1}{2}\right) E_1\left(x,y_4,\xi_1+\frac{1}{2}\right) \prod\limits_{m=1}^{k-1}P_-\left(x,y_4,\xi_{3m+1}\pm\frac{1}{2}\right)
\end{align*}
\item $n=3k+2$:
\begin{align*}
g_4(x,y_4)&:=E_1\left(x,y_4,\xi_1+\frac{1}{2}\right)E_1\left(x,y_4,\xi_n-\frac{1}{2}\right) \prod\limits_{m=1}^{k}P_-\left(x,y_4,\xi_{3m+1}\pm\frac{1}{2}\right)
\end{align*}
\end{itemize}

For $j=5$ we have:
\begin{itemize}
\item $n=3k$:
\begin{align*}
g_5(x,y_5)&:= E_2\left(x,y_5,\xi_1+\frac{3}{2}\right) E_1\left(x,y_5,\xi_n-\frac{1}{2}\right) \prod\limits_{m=1}^{k-1} P_-\left(x,y_5,\xi_{3m+2}\pm\frac{1}{2}\right)
\end{align*}
\item $n=3k+1$:
\begin{align*}
g_5(x,y_5)&:= E_2\left(x,y_5,\xi_1+\frac{3}{2}\right) E_2\left(x,y_5,\xi_n-\frac{3}{2}\right) \prod\limits_{m=1}^{k-1} P_-\left(x,y_5,\xi_{3m+2}\pm\frac{1}{2}\right)
\end{align*}
\item $n=3k+2$:
\begin{align*}
g_5(x,y_5)&:=P_-\left(x,y_5,\xi_n-\frac{1}{2}\right)
E_2\left(x,y_5,\xi_1+\frac{3}{2}\right) \prod\limits_{m=1}^{k-1} P_-\left(x,y_5,\xi_{3m+2}\pm\frac{1}{2}\right)
\end{align*}
\end{itemize}
This concludes the proof.
\end{proof}

\begin{remark}\label{rem:index-sets}
Listing the index sets $I_{j,\pm,n}$ for the smaller $n$ explicitly may help to convey the spirit of the construction. We list only the sets with $j \geq 3$, since $I_{1,\pm,n}$ and $I_{2,\pm,n}$ have a simpler pattern, see above.

For $n=3$:
\begin{align*}
I_{3,+,3} &:=\{1,2 \} \qquad &&I_{3,-,3} :=\{\ \}
\end{align*}

For $n=4$:
\begin{align*}
I_{3,+,4} &:=\{1,2 \} \qquad &&I_{3,-,4} :=\{4 \}\\
I_{4,+,4} &:=\{1,2,3 \} \qquad &&I_{4,-,4} :=\{\ \}
\end{align*}

For $n=5$:
\begin{align*}
I_{3,+,5} &:=\{1,2 \} \qquad &&I_{3,-,5} :=\{4,5 \}\\
I_{4,+,5} &:=\{1,2,3 \} \qquad &&I_{4,-,5} :=\{5 \}\\
I_{5,+,5} &:=\{1,2,3,4 \} \qquad &&I_{5,-,5} :=\{\ \}
\end{align*}

For $n=6$:
\begin{align*}
I_{3,+,6} &:=\{1,2,4,5 \} \qquad &&I_{3,-,6} :=\{\ \}\\
I_{4,+,6} &:=\{1,2,3 \} \qquad &&I_{4,-,6} :=\{5,6 \}\\
I_{5,+,6} &:=\{1,2,3,4 \} \qquad &&I_{5,-,6} :=\{6 \}
\end{align*}

For $n=7$:
\begin{align*}
I_{3,+,7} &:=\{1,2,4,5 \} \qquad &&I_{3,-,7} :=\{7\}\\
I_{4,+,7} &:=\{1,2,3,5,6 \} \qquad &&I_{4,-,7} :=\{\ \}\\
I_{5,+,7} &:=\{1,2,3,4 \} \qquad &&I_{5,-,7} :=\{6,7 \}
\end{align*}

For $n=8$:
\begin{align*}
I_{3,+,8} &:=\{1,2,4,5 \} \qquad &&I_{3,-,8} :=\{7,8 \}\\
I_{4,+,8} &:=\{1,2,3,5,6 \} \qquad &&I_{4,-,8} :=\{8 \}\\
I_{5,+,8} &:=\{1,2,3,4,6,7 \} \qquad &&I_{5,-,8} :=\{\ \}
\end{align*}
\end{remark}

\section{Stability of some heteroclinic cycles in DNN networks}
\label{sec_stability}

Without any claim to a comprehensive study, we determine stability results for some of the heteroclinic cycles in networks under the dynamics constructed above. The purpose is to illustrate how the choice of realization can affect the stability of the cycles. This may not be surprising but we believe it is noteworthy nevertheless\footnote{The reader less familiar with the study of the stability of cycles using basic transition matrices may want to focus on the results rather than the proofs.}. To do so, we first fix some more terminology and recall two definitions.

As is often done, we write $B_\varepsilon(x)$ for an $\varepsilon$-neighbourhood of a point $x \in \R^N$ and use $\ell(.)$ for standard Lebesgue measure. The basin of attraction of a compact invariant set $X \subset \R^N$ is the set of points in $\R^N$ with $\omega$-limit set in $X$, and we denote it by $\B(X)$. For $\delta >0$ the $\delta$-local basin of attraction $\B_\delta(X)$ is the subset of points in $\B(X)$ for which the forward trajectory never leaves $B_\delta(X)$.

\begin{definition}[\cite{Podvigina2013}, definition 3]
$X$ is called \emph{completely unstable} if there exists $\delta>0$ such that $\ell(\B_\delta(X))=0$.
\end{definition}

\begin{definition}[\cite{Podvigina2012}, definition 2]
$X$ is called \emph{fragmentarily asymptotically stable} if $\ell(\B_\delta(X))>0$ for any $\delta>0$.
\end{definition}

We now start with a result that applies to any quasi-simple cycle in a DNN network in $\R^4$. Quasi-simple cycles are defined in \cite{GdS-C2019}. These have one-dimensional connections in flow-invariant spaces of the same dimension. When restricting attention to coordinate planes, our construction produces quasi-simple cycles: inside $P_{0j}$ the connections are one-dimensional and $\dim P_{0j}=2$ for all $j$.

\begin{proposition}\label{prop:cu}
Let $\Sigma$ be a quasi-simple cycle with two nodes in a heteroclinic  network in $\R^4$ obtained from the construction above. If at each node there is a transverse positive eigenvalue for the Jacobian matrix at the node, then $\Sigma$ is completely unstable.
\end{proposition}

\begin{proof}
Cross-sections to the flow in $\R^4$ are two-dimensional and the global maps are the identity, that is, the connections are both of contracting-to-expanding type in the language of Garrido-da-Silva \cite{GdS-PhD}. Therefore, the basic transition matrices at the nodes (wlog, we denote the nodes by $\xi_1$ and $\xi_2$) are of the form
$$
M_1 = \left[
\begin{array}{cc}
\dfrac{c_{12}}{e_{12}} & 0 \\
 & \\
-\dfrac{t_1}{e_{12}} & 1
\end{array}
\right] \;\; \mbox{ and  } \;\; 
M_2 = \left[
\begin{array}{cc}
\dfrac{c_{21}}{e_{21}} & 0 \\
 & \\
-\dfrac{t_2}{e_{21}} & 1
\end{array}
\right] .
$$
The matrix $M_j$ describes the transition $H_j^{\text{in},k} \rightarrow H_k^{\text{in},j}$ with $j\neq k \in \{1,2\}$. Because we assume that the transverse eigenvalue is positive, we can write the transition matrices $M^{(j)}=M_kM_j$ as follows
$$
M^{(j)} = \left[
\begin{array}{cc}
\alpha_j & 0 \\
 & \\
\beta_j & 1
\end{array}
\right] 
$$
with $\alpha_j>0$ and $\beta_j<0$.
The eigenvalues of $M^{(j)}$ are $\alpha_j$ and 1. According to Lemma 5 in Podvigina \cite{Podvigina2012}, a cycle is completely unstable if one of the conditions in the lemma fails. We show that if $\lambda_{\max}>1$ (condition (ii) is satisfied) then $w_1^{\max}w_2^{\max}<0$ (condition (iii) fails) and hence, the cycle is completely unstable.

Assume then that $\alpha_j>1$ so that $\alpha_j = \lambda_{\max} >1$. Calculating the eigenvector associated with $\lambda_{\max}$ we obtain
$$
w_2^{\max} = \dfrac{\beta_j}{\alpha_j -1} w_1^{\max}.
$$
Given that $\alpha_j>1$ and $\beta_j<0$ the proof finishes.
\end{proof}

Our DNN network with three nodes has the same digraph as the Rock-Scissors-Paper network realized using replicator dynamics in \cite{GdS-C2019,GdS-C2020}.

An immediate consequence of Proposition \ref{prop:cu} is that the choice of realisation of the Rock-Scissors-Paper network produces very different results concerning the stability of the 2-node cycles. The stability of the cycles in the Rock-Scissors-Paper network under replicator dynamics may be found in the work of Garrido-da-Silva and Castro \cite{GdS-C2019,GdS-C2020} where the 2-node cycles are shown to possess some stability for various parameter values. Under replicator dynamics there are distinct parameter values for which each of the 2-node cycles is fragmentarily asymptotically stable. With our realisation all 2-node cycles are completely unstable\footnote{Since our construction is abstract, it is meaningless to relate our parameters to those used in \cite{GdS-C2019,GdS-C2020}, which are associated to the players' payoffs.}. It should be noted that Proposition \ref{prop:cu} does not apply to the model under replicator dynamics studied in \cite{GdS-C2019,GdS-C2020} as the connections are then of contracting-to-transverse type leading to basic transition matrices unlike those in Proposition \ref{prop:cu}.

In higher dimensions the stability of 2-node cycles can exhibit more varied stability properties in DNN networks. As an example, we study the stability of the 2-node cycle $C_{13} = [\xi_1 \rightarrow \xi_3 \rightarrow \xi_1]$ in the 4-node network in $\R^5$, see subsection~\ref{subsec:4}.

Denote the eigenvalues of the Jacobian matrix at $\xi_i$ in the direction of $y_j$ by $-c_{ij}$, if they are negative, and by $e_{ij}$ if they are positive. In the direction of $x$, the eigenvalues are denoted by $-r_i$. Then we can write
\begin{align*}
J(\xi_1) & =  \diag(-r_1, -c_{11}, e_{12}, e_{13}, -c_{14}) \\
J(\xi_3) & =  \diag(-r_3, e_{31}, -c_{32}, -c_{33}, e_{34}).
\end{align*}
Since the connections occur in the subspace of coordinates $(x,y_1,y_3)$ the directions $y_2$ and $y_4$ are transverse to this cycle. The sign of the eigenvalues in these two directions is determined by the other connections in the network that involve the nodes $\xi_1$ or $\xi_3$.
In our construction we choose the eigenvalues to be negative whenever possible, that is, transverse eigenvalues are negative unless there is a connection from the equilibrium in the direction of the corresponding eigenvector.
Following \cite{Podvigina2012} we construct the basic transition matrices $M_i$, describing the dynamics between two consecutive incoming cross-sections $H^{\text{in}}_i \rightarrow H^{\text{in}}_{i+1}$
\begin{align*}
M_1 =  \left( \begin{array}{ccc}
\dfrac{c_{11}}{e_{13}} & 0 & 0 \\
 & & \\
-\dfrac{e_{12}}{e_{13}} & 1 & 0 \\
 & & \\
\dfrac{c_{14}}{e_{13}} & 0 & 1 
\end{array}\right)
& \hspace{.3cm} \mbox{and} \hspace{.5cm}
M_3 =  \left( \begin{array}{ccc}
\dfrac{c_{33}}{e_{31}} & 0 & 0 \\
 & & \\
\dfrac{c_{32}}{e_{31}}  & 1 & 0 \\
 & & \\
-\dfrac{e_{34}}{e_{31}}  & 0 & 1 
\end{array}\right).
\end{align*}
From these, we obtain the transition matrices $M^{(i)}$ describing the return map to $H^{\text{in}}_i$ as follows
\begin{align*}
M^{(1)} =  \left( \begin{array}{ccc}
\dfrac{c_{11}c_{33}}{e_{13}e_{31}} & 0 & 0 \\
 & & \\
-\dfrac{e_{12}c_{33}}{e_{13}e_{31}}+\dfrac{c_{32}}{e_{31}}  & 1 & 0 \\
 & & \\
\dfrac{c_{14}c_{33}}{e_{13}e_{31}}-\dfrac{e_{34}}{e_{31}} & 0 & 1 
\end{array}\right)
& \hspace{.2cm} \mbox{and} \hspace{.4cm}
M^{(3)} =  \left( \begin{array}{ccc}
\dfrac{c_{11}c_{33}}{e_{13}e_{31}} & 0 & 0 \\
 & & \\
\dfrac{c_{11}c_{32}}{e_{13}e_{31}}-\dfrac{e_{12}}{e_{13}}  & 1 & 0 \\
 & & \\
-\dfrac{c_{11}e_{34}}{e_{13}e_{31}}+\dfrac{c_{14}}{e_{13}}  & 0 & 1 
\end{array}\right).
\end{align*}

We study the stability of $C_{13}$ by checking \cite[Lemma 3]{Podvigina2012}. Let $\lambda_{\max}$ be the maximum, in absolute value, eigenvalue of the matrix $M^{(i)}$ and $\boldsymbol{w}^{\max}=\left(w_1^{\max},\ldots,w_3^{\max}\right)$ be an associated eigenvector. The stability index along the connection $[\xi_i \rightarrow \xi_j]$ is bigger than $-\infty$ if and only if the three following conditions are satisfied:
\begin{enumerate}
\item[(i)] $\lambda_{\max}$ is real;
\item[(ii)] $\lambda_{\max}>1$;
\item[(iii)] $w_{l}^{\max}w_{q}^{\max}>0$ for all $l,q=1,\ldots N$.
\end{enumerate}

The first condition is trivially satisfied.
We have $\lambda_{\max}>1 \Leftrightarrow c_{11}c_{33}>e_{13}e_{31}$. The reader may check that the coordinates of $\boldsymbol{w}^{\max}$ have the same sign if and only if 
\begin{align*}
-\dfrac{e_{12}c_{33}}{e_{13}e_{31}}+\dfrac{c_{32}}{e_{31}} & > 0 \\
\dfrac{c_{14}c_{33}}{e_{13}e_{31}}-\dfrac{e_{34}}{e_{31}} & > 0 \\
\dfrac{c_{11}c_{32}}{e_{13}e_{31}}-\dfrac{e_{12}}{e_{13}}  & > 0 \\
-\dfrac{c_{11}e_{34}}{e_{13}e_{31}}+\dfrac{c_{14}}{e_{13}} & > 0.
\end{align*}

Hence, if one of the above conditions is not satisfied the cycle $C_{13}$ is completely unstable. Otherwise, it is fragmentarily asymptotically stable. A more detailed study of the stability of the cycles can be done by using Lohse \cite[Theorem 3.1]{Lohse2015} and Garrido-da-Silva and Castro \cite[Lemma 3.2]{GdS-C2019} to establish that a cycle is fragmentarily asymptotically stable by relating the sign of the stability indices to the stability properties. The stability index along each connection can be calculated using \cite[Theorem 3.10]{GdS-C2019} and the function $F^{\ind}$ defined in \cite[Appendix A.1]{GdS-C2019}. Since the study of stability is not the main purpose of this work, we leave this for the interested reader. The point we want to make is that, as expected, different construction methods lead to different stability properties.

\paragraph{Acknowledgements:}

The first author was partially supported by CMUP, member of LASI, which is financed by national funds through FCT – Funda\c{c}\~ao para a Ci\^encia e a Tecnologia, I.P., under the projects with reference UIDB/00144/2020 and UIDP/00144/2020.

Both authors are grateful to the reviewers whose comments helped improve the exposition.

\end{document}